\documentclass[12pt]{amsart}
\usepackage{amscd,amssymb,hyperref}
\usepackage{graphicx}
\usepackage{graphics}
\usepackage{epsfig}
\usepackage{picinpar}
\usepackage{wrapfig}
\usepackage{amsmath}
\usepackage{color}
\usepackage{pstricks,pst-node,pst-text,pst-3d}
\usepackage{pst-plot}
\usepackage{verbatim}

\newtheorem{thm}{Theorem}[section]
\newtheorem{lem}[thm]{Lemma}
\newtheorem{cor}[thm]{Corollary}
\newtheorem{prop}[thm]{Proposition}

\def\square{\vbox{
      \hrule height 0.4pt
      \hbox{\vrule width 0.4pt height 5.5pt \kern 5.5pt \vrule width 0.4pt}
      \hrule height 0.4pt}}

\def\ch\mathrm{c h}

\newcommand{\VAP}{\mathrm{VAP}}

\let\la=\langle
\let\ra=\rangle

\numberwithin{equation}{section}

\begin{document}

\title[On virtual cabling]{On virtual cabling and  structure of $4$-strand virtual pure braid group}

\author[V. Bardakov]{Valeriy G. Bardakov}
\address{Sobolev Institute of Mathematics, Novosibirsk 630090, Russia,}
\address{Novosibirsk State University, Novosibirsk 630090, Russia,}
\address{Novosibirsk State Agrarian University, Dobrolyubova street, 160,  Novosibirsk 630039, Russia,}
\email{bardakov@math.nsc.ru}

%\author{Roman Mikhailov}
%\address{Laboratory of Modern Algebra and Applications, St. Petersburg State University, 14th Line, 29b,
%Saint Petersburg, 199178 Russia and St. Petersburg Department of
%Steklov Mathematical Institute} \email{rmikhailov@mail.ru}

\author{Jie Wu }
\address{Department of Mathematics, National University of Singapore, 10 Lower Kent Ridge Road, Singapore 119076} \email{matwuj@nus.edu.sg}
\urladdr{www.math.nus.edu.sg/\~{}matwujie}
\thanks{The main result is supported by the Russian Science Foundation grant N 16-11-10073.}

\begin{abstract}
This article is dedicate to cabling on virtual braids. This construction gives a new generating set for the virtual pure braid group $VP_n$. Consequently we describe $VP_4$ as HNN-extension. As an application to classical braids, we find a new presentation of the Artin pure braid group $P_4$ in terms of the cabled generators.

\end{abstract}
\subjclass[2010]{20F36, 55Q40, 18G30}
\keywords{homotopy group, virtual braid group, simplicial group, virtual cabling}

\maketitle

%\section{Introduction} \label{int}

\section{Introduction} \label{int}

Cabling on classical braids has been used for establishing the fundamental connections between the homotopy groups and the theory of Brunnian braids~\cite{BCWW} as well as a relationship between associators (for quasi-triangular quasi-Hopf algebras) and (a variant of) the Grothendieck-Teichmuller group~\cite{Bar-Natan}. In the paper \cite{BM}, cabling for braids was used to study some properties of Burau representation. Similar operation (called \textit{naive cabling}) on framed links has been explored in~\cite{LLW} with obtaining simplicial groups arising from link groups.

The purpose of this article is to explore cabling for virtual braids. Along the ideas in~\cite{CW} on cabling for classical braids, one gets cabling operation for virtual pure braid group $VP_n$ that gives new generators for $VP_n$. More precisely, for $n\geq 3$, the group $VP_n$ is generated by the $n$-strand virtual braids obtained by taking $(k,l)$-cabling on the standard generators $\lambda_{1,2}$ and $\lambda_{2,1}$ of $VP_2$ together with adding trivial strands $n-k-l$ to the end for $1\leq k\leq n-1$ and $2\leq k+l\leq n$, where a $(k,l)$-cabling on a $2$-strand virtual braid means to take $k$-cabling on the first strand and $l$-cabling on the second strand.

Different from the classical situation~\cite{CW} that the $n$-strand braids cabled from the standard generator $A_{1,2}$ for $P_2$ generates a free group of rank $n-1$, the subgroup of $VP_n$ generated by $n$-strand virtual braids cabled from $\lambda_{1,2}$ and $\lambda_{2,1}$, which is denoted by $T_{n-1}$, is no longer free for $n\geq3$.
%The group $T_{n-1}$ has $2(n-1)$ generators $a_{k,n-k}$ and $b_{k,n-k}$, $1\leq k\leq n-1$, where $a_{k,n-k}$ (or $b_{k,n-k}$) is obtained by taking $k$-cabling on the first strand and $(n-k)$-cabling on the second strand of $\lambda_{1,2}$ (or $\lambda_{2,1}$).
For the first nontrivial case that $n=3$, a presentation of $T_2$ has been explored with producing a decomposition theorem for $VP_3$ using cabled generators~\cite{BMVW}.

Our main work in this article is to introduce a new generating set for $VP_n$, define a simplicial group $T_*$ and extend the results on $VP_3$ in~\cite{BMVW} to $VP_4$. The main result is Theorem~\ref{t4.11} that describe $VP_4$ as HNN-extension. As a consequence, we get a presentation for the group $T_3$ in Theorem~\ref{T_3}. In the next article \cite{BMW} we prove the lifting theorem for the virtual braids. From this theorem follows that if we know the structure of $VP_4$, $T_3$ or $P_4$, then using degeneracy maps we can find the structure of $VP_n$, $T_n$ or $P_n$ for all bigger $n$.

The article is organized as follows. In Section \ref{virt}, we give a review on braid groups and virtual braid groups. The simplicial structure on virtual pure braid groups will be discussed in Section~\ref{simplicial}. In Section ~\ref{s41}, we discuss the cabling operation on classical pure braid group $P_n$ as subgroup $VP_n$.  In particular, we give a new presentation of the Artin pure braid group $P_4$ in terms of the cabled generators in Proposition~\ref{prop4.1}. We explore the structures of $VP_3$ and $VP_4$ in Section~\ref{s4}.

\subsection{Acknowledgements}
This article was written when the first author visited  College of Mathematics and information Science Hebei Normal University. He thanks the administration for good working conditions.  The authors would like to thank Roman Mikhailov for interesting ideas and useful discussion and  Yu. Mikhal'chishina, who made the picture.

\section{Braid and virtual braid groups} \label{virt}

\subsection{Braid group} The braid group $B_n$ on $n$ strings  is generated by $\sigma_1,\,  \sigma_2, \, \ldots , \, \sigma_{n-1}$ and is defined by relations
\begin{align*}
& \sigma_i \sigma_{i+1} \sigma_i = \sigma_{i+1} \sigma_i \sigma_{i+1},\\
& \sigma_i \sigma_{j} = \sigma_j \sigma_{i},~~|i-j|>1.
%\label{relation},
\end{align*}

Let $S_n$, $n \geq 1$ be the symmetric group which is generated by $\rho_1, \, \rho_2, \, \ldots , \, \rho_{n-1}$ and is defined by relations
\begin{align*}
& \rho_i^2 = 1,~~~i = 1, 2, \ldots, n-1,\\
& \rho_i \rho_{i+1} \rho_i = \rho_{i+1} \rho_i \rho_{i+1},~~~i = 1, 2, \ldots, n-2,\\
& \rho_i \rho_{j} = \rho_j \rho_{i},~~|i-j|>1.
%\label{relation},
\end{align*}

There is a homomorphism $B_n \to S_n$, which sends $\sigma_i$ to $\rho_i$. Its kernel is the pure braid group $P_n$. This group is generated by elements $A_{i,j}$, $1 \leq i < j \leq n$, where
$$
A_{i,i+1} = \sigma_i^2,
$$
$$
A_{i,j} = \sigma_{j-1} \sigma_{j-2} \ldots \sigma_{i+1} \sigma_i^2 \sigma_{i+1}^{-1} \ldots \sigma_{j-2}^{-1} \sigma_{j-1}^{-1},~~~i+1 < j \leq n,
$$
and is defined by relations (where $\varepsilon = \pm 1$):
\begin{align*}
& A_{ik}^{-\varepsilon} A_{kj} A_{ik}^{\varepsilon} = (A_{ij} A_{kj})^{\varepsilon} A_{kj} (A_{ij} A_{kj})^{-\varepsilon},\\
& A_{km}^{-\varepsilon} A_{kj} A_{km}^{\varepsilon} = (A_{kj} A_{mj})^{\varepsilon} A_{kj} (A_{kj} A_{mj})^{-\varepsilon},~~m < j, \\
& A_{im}^{-\varepsilon} A_{kj} A_{im}^{\varepsilon} = [A_{ij}^{-\varepsilon}, A_{mj}^{-\varepsilon}]^{\varepsilon} A_{kj} [A_{ij}^{-\varepsilon}, A_{mj}^{-\varepsilon}]^{-\varepsilon}, ~~i < k < m,\\
& A_{im}^{-\varepsilon} A_{kj} A_{im}^{\varepsilon} = A_{kj}, ~~k < i, m < j~\mbox{or}~ m < k,
\end{align*}
Here and further  $[a,b] = a^{-1} b^{-1} a b$ is the commutator of $a$ and $b$, $a^b = b^{-1} a b$ is the conjugation of $a$ by $b$.

There is an epimorphism of $P_n$ to $P_{n-1}$ what is removing of the $n$-th string. Its kernel $U_n = \langle A_{1n}, A_{2n}, \ldots, A_{n-1,n} \rangle$ is a free group of rank $n-1$ and $P_n = U_n \leftthreetimes P_{n-1}$ is a semi-direct product of $U_n$ and $P_{n-1}$. Hence,
$$
P_n = U_n \leftthreetimes (U_{n-1} \leftthreetimes (\ldots \leftthreetimes (U_3 \leftthreetimes U_2)) \ldots ),
$$
is a semi-direct product of free groups and $U_2 = \langle A_{12}\rangle$ is the infinite cyclic  group.

\subsection{Virtual braid group} \label{virt}

The virtual braid group $VB_n$ is generated by elements
$$
\sigma_1,\,  \sigma_2, \, \ldots , \, \sigma_{n-1}, \, \rho_1, \, \rho_2, \, \ldots , \, \rho_{n-1},
$$
where $\sigma_1,\,  \sigma_2, \, \ldots , \, \sigma_{n-1}$ generate the classical braid group $B_n$ and
the elements $\rho_1$,  $\rho_2$,  $\ldots $,  $\rho_{n-1}$ generate the symmetric group
$S_n$. Hence, $VB_n$ is defined by relations of $B_n$, relations of $S_n$
and mixed relations:
$$
\sigma_i \rho_j = \rho_j \sigma_i,~~~|i-j| > 1,
$$
$$
\rho_i \rho_{i+1} \sigma_i = \sigma_{i+1} \rho_i \rho_{i+1}.
$$

As for the classical braid groups there exists the canonical
epimorphism of $VP_n$ onto the symmetric group $VB_n\to S_n$ with the
kernel called the {\it virtual pure  braid group} $VP_n$. So we have a
short exact sequence
\begin{equation*}
1 \to VP_n \to VB_n \to S_n \to 1.
\end{equation*}
Define the following elements in $VP_n$:
$$
\lambda_{i,i+1} = \rho_i \, \sigma_i^{-1},~~~
\lambda_{i+1,i} = \rho_i \, \lambda_{i,i+1} \, \rho_i = \sigma_i^{-1} \, \rho_i,
~~~i=1, 2, \ldots, n-1,
$$
$$
\lambda_{ij} = \rho_{j-1} \, \rho_{j-2} \ldots \rho_{i+1} \, \lambda_{i,i+1} \, \rho_{i+1}
\ldots \rho_{j-2} \, \rho_{j-1},
$$
$$
\lambda_{ji} = \rho_{j-1} \, \rho_{j-2} \ldots \rho_{i+1} \, \lambda_{i+1,i} \, \rho_{i+1}
\ldots \rho_{j-2} \, \rho_{j-1}, ~~~1 \leq i < j-1 \leq n-1.
$$
It is shown in \cite{B} that the group $VP_n\ (n\geq 2)$ admits a
presentation with the  generators $\lambda_{ij},\ 1\leq i\neq j\leq n,$
and the following relations:
\begin{align}
& \lambda_{ij}\lambda_{kl}=\lambda_{kl}\lambda_{ij} \label{rel},\\
&
\lambda_{ki}\lambda_{kj}\lambda_{ij}=\lambda_{ij}\lambda_{kj}\lambda_{ki}
\label{relation},
\end{align}
where distinct letters stand for distinct indices.

Like the classical pure braid groups, groups $VP_n$ admit a
semi-direct product decompositions \cite{B}: for $n\geq 2,$ the
$n$-th virtual pure braid group can be decomposed as
\begin{equation}
VP_n=V_{n-1}^*\rtimes VP_{n-1},~~n \geq 2,
\label{eq:s_d_dec}
\end{equation}
where $V_{n-1}^*$ is a  subgroup of $VP_{n}$, $V_1^* = F_2$, $VP_1$ is supposed
to be the trivial group.

\section{Simplicial  groups $VP_*$ and $T_*$}\label{simplicial}
\subsection{Simplicial sets and simplicial groups} Recall the definition of simplicial groups (see \cite[p.~300]{MP} or \cite{BCWW}). A sequence of sets $\mathcal{X} = \{ X_n \}_{n \geq 0}$  is called a
{\it simplicial set} if there are face maps:
$$
d_i : X_n \longrightarrow X_{n-1} ~\mbox{for}~0 \leq i \leq n
$$
and degeneracy  maps
$$
s_i : X_n \longrightarrow X_{n+1} ~\mbox{for}~0 \leq i \leq n,
$$
that are  satisfy the following simplicial identities:
\begin{enumerate}
\item $d_i d_j = d_{j-1} d_i$ if $i < j$,
\item $s_i s_j = s_{j+1} s_i$ if $i \leq j$,
\item $d_i s_j = s_{j-1} d_i$ if $i < j$,
\item $d_j s_j = id = d_{j+1} s_j$,
\item $d_i s_j = s_{j} d_{i-1}$ if $i > j+1$.
\end{enumerate}

\subsection{The cablings of virtual pure braid groups}
By using the same ideas in the work~\cite{BCWW,CW} on the classical braids, we have a simplcial group
$$
\VAP_* :\ \ \ \ldots\ \begin{matrix}\longrightarrow\\[-3.5mm] \ldots\\[-2.5mm]\longrightarrow\\[-3.5mm]
\longleftarrow\\[-3.5mm]\ldots\\[-2.5mm]\longleftarrow \end{matrix}\ VP_4 \ \begin{matrix}\longrightarrow\\[-3.5mm]\longrightarrow\\[-3.5mm]\longrightarrow\\[-3.5mm]\longrightarrow\\[-3.5mm]\longleftarrow\\[-3.5mm]
\longleftarrow\\[-3.5mm]\longleftarrow
\end{matrix}\ VP_3\ \begin{matrix}\longrightarrow\\[-3.5mm] \longrightarrow\\[-3.5mm]\longrightarrow\\[-3.5mm]
\longleftarrow\\[-3.5mm]\longleftarrow \end{matrix}\ VP_2\ \begin{matrix} \longrightarrow\\[-3.5mm]\longrightarrow\\[-3.5mm]
\longleftarrow \end{matrix}\ VP_1$$
on pure virtual braid groups with $\VAP_n=VP_{n+1}$, the face homomorphism
$$
d_i : \VAP_n=VP_{n+1} \longrightarrow \VAP_{n-1}=VP_n
$$
given by deleting $(i+1)$th strand for $0\leq i\leq n$, and the degeneracy homomorphism
$$
s_i : \VAP_n=VP_{n+1} \longrightarrow \VAP_{n+1}=VP_{n+2}
$$
given by doubling the $(i+1)$th strand for $0\leq i\leq n$.

The idea of cabling  is obtained from the geometric description, which can be regarded as the formal definition. See the  Figure 1.

%\begin{figure}[h]
%\noindent\centering{\includegraphics[height=0.3 \textwidth]{valeriy-file-1}}
%\caption{Degeneracy map $s_1$ }
%%\label{lambda}
%\end{figure}

\begin{figure}[h]
\noindent\centering{\includegraphics[height=0.3\textwidth]{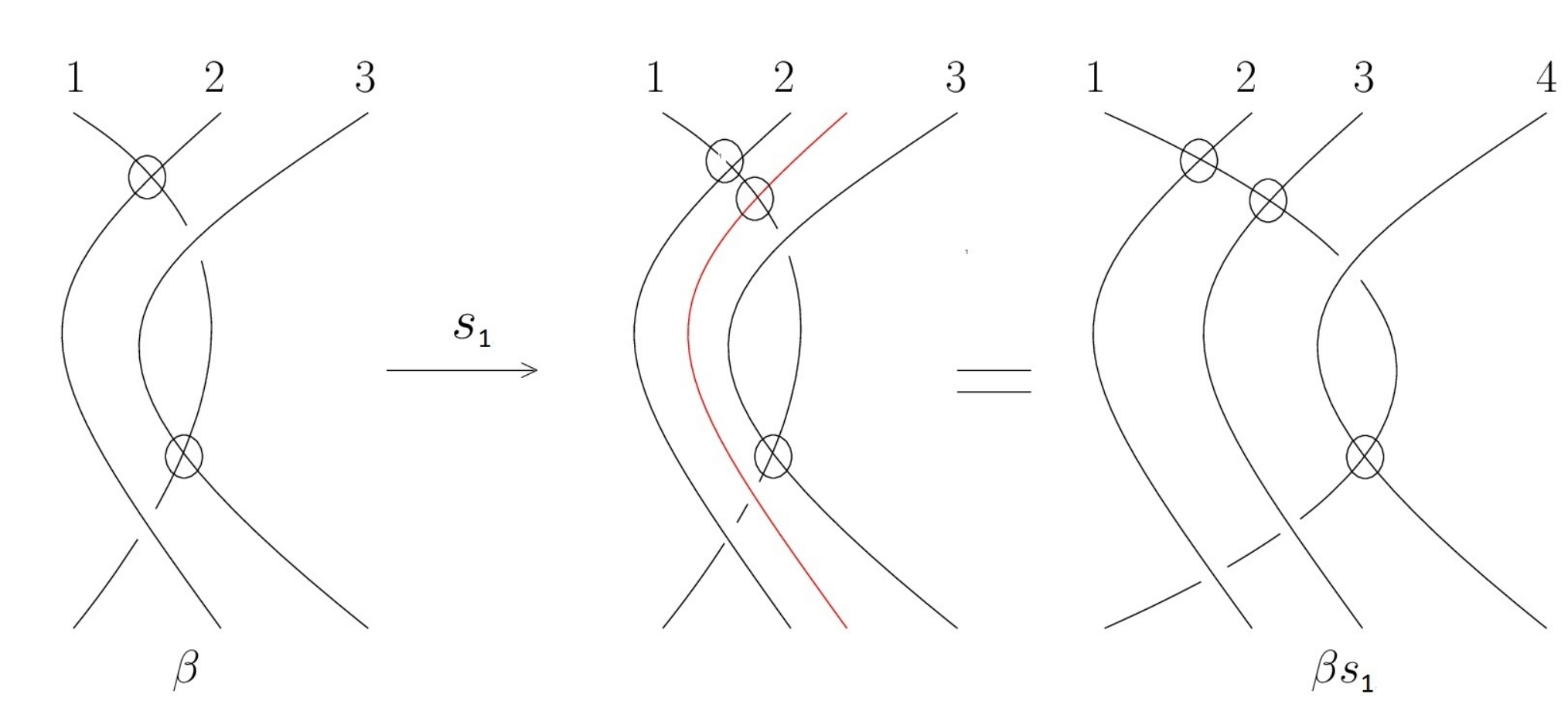}}
\caption{Degeneracy map $s_1$ }
%\label{lambda}
\end{figure}

The proof of the following proposition is straightforward.

\begin{prop} \label{p3.1}
The sequence of groups $\VAP_*$ with $\VAP_n=VP_{n+1}$ for $n\geq0$ is a simplicial group under the faces
$d_i : \VAP_{n-1}=VP_{n} \longrightarrow \VAP_{n-2}=VP_{n-1}$, $0\leq i\leq n-1$, and
 degeneracies $s_i : \VAP_{n-1}=VP_{n} \longrightarrow \VAP_{n}=VP_{n+1}$, $0\leq i\leq n-1$, given the group homomorphism with acting on the generators $\lambda_{k,l}$ and $\lambda_{l,k}$, $1 \leq k < l \leq n$, of $VP_{n}$ by the rules
$$
s_i (\lambda_{k,l}) = \left\{
\begin{array}{lcl}
\lambda_{k+1,l+1} & \textrm{ for }&i < k-1,\\
\lambda_{k,l+1} \lambda_{k+1,l+1} & \textrm{ for } &i = k-1, \\
\lambda_{k,l+1} & \textrm{ for } &k-1 < i < l-1,\\
& \\
\lambda_{k,l+1} \, \lambda_{k,l} & \textrm{ for }& i = l-1, \\
& \\
\lambda_{k,l} & \textrm{ for } &i > l-1,
\end{array}
\right.
$$
$$
s_i (\lambda_{l,k}) = \left\{
\begin{array}{lcl}
\lambda_{l+1,k+1} &\textrm{ for }&i < k-1,\\
\lambda_{l+1,k+1} \lambda_{l+1,k} &\textrm{ for }&i = k-1, \\
\lambda_{l+1,k} & \textrm{ for }&k-1 < i < l-1,\\
& \\
\lambda_{l,k} \, \lambda_{l+1,k}   & \textrm{ for }&i = l-1, \\
& \\
\lambda_{l,k} & \textrm{ for }&i > l-1,
\end{array}
\right.
$$
$$
d_i (\lambda_{k,l}) = \left\{
\begin{array}{lcl}
\lambda_{k-1,l-1} & \textrm{ for } &0 \leq i < k-1, \\
1 & \textrm{ for } & i =k-1,\\
\lambda_{k,l-1} & \textrm{ for } &k-1 < i < l-1, \\
1 & \textrm{ for } & i =l,\\
\lambda_{k,l} & \textrm{ for } &l-1 < i \leq n-1, \\
\end{array}
\right.
$$
$$
d_i (\lambda_{l,k}) = \left\{
\begin{array}{lcl}
\lambda_{k-1,l-1} & \textrm{ for } &0 \leq i < k-1, \\
1 & \textrm{ for } & i = k-1,\\
\lambda_{l-1,k} & \textrm{ for } &k-1 < i < l-1, \\
1 & \textrm{ for } & i = l-1,\\
\lambda_{l,k} & \textrm{ for } &l-1 < i \leq n-1. \\
\end{array}
\right.
$$
 \hfill $\Box$
\end{prop}

Let $T_*$ be the smallest simplicial subgroup of $\VAP_*$ with the $1$-simplex group $T_1=\VAP_1=VP_2$. It is routine to see that the group $T_n$ as a subgroup of $VP_{n+1}$ can be constructed recursively as follows:
\begin{enumerate}
\item[] $T_0=\{1\}$, $T_1=VP_2$, and $T_{n+1} = \langle s_0(T_n), s_1(T_n), \ldots,
s_{n}(T_n) \rangle. $
\end{enumerate}

Let
\begin{equation}\label{a_{kl}}
\begin{array}{rcl}
a_{k,n+1-k}
&=&s_{n-1}s_{n-2}\cdots s_{k}\hat{s}_{k-1}s_{k-2}\cdots s_0\lambda_{1,2},\\
\end{array}
\end{equation}
\begin{equation}\label{b_{kl}}
\begin{array}{rcl}
b_{k,n+1-k}
&=&s_{n-1}s_{n-2}\cdots s_{k}\hat{s}_{k-1}s_{k-2}\cdots s_0\lambda_{2,1}\\
\end{array}
\end{equation}
be the elements in $VP_{n+1}$ for $1\leq k\leq n$.
By direct computations,  we have the following formulae.
\begin{equation}\label{equation3.1}
a_{n-k,k} = \left\{
\begin{array}{lr}
\lambda_{1n} \lambda_{2n} \ldots \lambda_{n-1,n} & for  ~k = 1,\\
\lambda_{1n} \lambda_{2n} \ldots \lambda_{n-k,n} a_{n-k,k-1} & for ~1 < k < n, \\
\lambda_{1n} a_{1,n-2} & for  ~k = n-1,
\end{array}
\right.
\end{equation}
\begin{equation}\label{equation3.2}
b_{n-k,k} = \left\{
\begin{array}{lr}
\lambda_{n,n-1} \lambda_{n,n-2} \ldots \lambda_{n1} & for  ~k = 1,\\
 b_{n-k,k-1} \lambda_{n,n-k} \lambda_{n,n-k-1} \ldots \lambda_{n1} & for ~1 < k < n, \\
b_{1,n-2}  \lambda_{n1} & for  ~k = n-1,
\end{array}
\right.
\end{equation}
Moreover the generators $\lambda_{ij}$ can be written in terms of $a_{k,l}$ and $b_{s,t}$ as follows:
\begin{equation}\label{equation3.3}
\lambda_{kn} = \left\{
\begin{array}{lr}
a_{1,n-1} \, a_{1,n-2}^{-1} & for  ~k = 1,\\
a_{k-1,n-k} \, a_{k-1,n-k+1}^{-1} \, a_{k,n-k} \, a_{k,n-k-1}^{-1} & for ~1 < k < n, \\
a_{n-2,1} \, a_{n-2,2}^{-1} \, a_{n-1,1} & for  ~k = n-1,
\end{array}
\right.
\end{equation}
\begin{equation}\label{equation3.4}
\lambda_{nk} = \left\{
\begin{array}{lr}
b_{1,n-2}^{-1} \, b_{1,n-1} & for  ~k = 1,\\
b_{k,n-k-1}^{-1} \, b_{k,n-k} \, b_{k-1,n-k+1}^{-1} \, b_{k-1,n-k} & for ~1 < k < n, \\
b_{n-1,1} \, b_{n-2,2}^{-1} \, b_{n-2,1} & for  ~k = n-1,
\end{array}
\right.
\end{equation}

From the above formulae, we have the following proposition.

\begin{prop} \label{p3.3}
Consider $VP_k$ as a subgroup of $VP_{k+1}$ by adding a trivial strand in the end. Then
\begin{enumerate}
\item  The subgroup $T_{n-1}$ of $VP_n$, $n \geq 3$, is generated by elements $a_{k,l}$, $b_{k,l}$, $k+l = n$.
\item The group $VP_n=\la T_1, T_2,\ldots, T_{n-1}\ra$ generated by $a_{k,l}$ and $b_{k,l}$ for $2\leq k+l\leq n, 1\leq k,l\leq n-1$.
\item $VP_{n+1}=\la VP_n, s_0VP_n, s_1 VP_n,\ldots, s_{n-1}VP_n \ra$ for $n\geq 2$.\hfill $\Box$
\end{enumerate}
\end{prop}

Let $c_{ij} = b_{ij} a_{ij}$. Put
$$
T_i^c = \langle c_{ij}~|~i+j = n-1  \rangle,~~i = 1, 2, \ldots, n-1.
$$
Notice that
$$
c_{1,1}=b_{1,1}a_{1,1}=\lambda_{2,1}\lambda_{1,2}=\sigma_1^{-1}\rho_1\rho_1\sigma_1^{-1}=\sigma_1^{-2}
$$
is a generator for $P_2$ as a subgroup of $VP_2$. The cabled braid $c_{i,j}$ lies in $P_{i+j+1}\leq VP_{i+j+1}$. It is straightforward to see that the following proposition holds for classical braids.

\begin{prop}\label{proposition3.3}
Consider $P_k$ as a subgroup of $P_{k+1}$ by adding a trivial strand in the end. Then
\begin{enumerate}
\item  The subgroup $T^c_{n-1}$ of $P_n$, $n \geq 3$, is generated by elements $c_{k,l}$, $k+l = n$.
\item The group $P_n=\la T^c_1, T^c_2,\ldots, T^c_{n-1}\ra$ generated by $c_{k,l}$ for $2\leq k+l\leq n, 1\leq k,l\leq n-1$.
\item $P_{n+1}=\la P_n, s_0P_n, s_1 P_n,\ldots, s_{n-1}P_n \ra$ for $n\geq 2$.\hfill $\Box$
\end{enumerate}
\end{prop}

\section{Cabling of the classical pure braid group} \label{s41}

In the present section we find a set of defining relations of $P_4$ in the generators $c_{ij}$ in Proposition~\ref{proposition3.3}.
%From results of the next section will follow the general case.

\begin{prop}\label{prop4.1} The group $P_4$ is generated by elements
$$
c_{11},~~ c_{21},~~ c_{12},~~ c_{31},~~c_{22},~~ c_{13}
$$
and is defined by relations (where $\varepsilon = \pm 1$):
$$
c_{21}^{c_{11}^{\varepsilon}} = c_{21},~~~c_{12}^{c_{11}^{\varepsilon}} = c_{12}^{c_{21}^{-\varepsilon}},~~~c_{31}^{c_{11}^{\varepsilon}} = c_{31},~~~c_{22}^{c_{11}^{\varepsilon}} = c_{22},~~~c_{13}^{c_{11}^{\varepsilon}} = c_{13}^{c_{22}^{-\varepsilon}},
$$
$$
c_{31}^{c_{21}^{\varepsilon}} = c_{31},~~~c_{22}^{c_{21}^{\varepsilon}} = c_{22}^{c_{31}^{-\varepsilon}},~~~c_{13}^{c_{21}^{\varepsilon}} = c_{13}^{c_{22}^{\varepsilon} c_{31}^{-\varepsilon}},
$$
$$
c_{31}^{c_{12}^{\varepsilon}} = c_{31},~~~c_{13}^{c_{12}^{\varepsilon}} = c_{13}^{c_{31}^{-\varepsilon}}.
$$
$$
c_{22}^{c_{12}^{-1}} = [c_{31}, c_{13}^{-1}] \, [c_{13}^{-1}, c_{22}] \, c_{22} \, [c_{21}^2, c_{12}^{-1}] = c_{13}^{c_{31}} c_{13}^{-c_{22}}  c_{22} [c_{21}^2, c_{12}^{-1}],
$$
$$
c_{22}^{c_{12}} = [c_{12}, c_{21}^{-2}] \, c_{22} \, [c_{22}^{-3}, c_{13}]  \, [c_{13}, c_{31}^{-1}] = [c_{12}, c_{21}^{-2}] \, c_{13}^{-c_{22}^{-2}} \, c_{22} \, c_{13}^{c_{31}^{-1}} .
$$

\end{prop}

\begin{proof}
Rewrite the generators $c_{ij}$ in the standard generators of $P_4$. We have
$$
c_{11} = b_{11} a_{11} = \lambda_{21} \lambda_{12} = \sigma_1^{-1} \rho_1 \rho_1 \sigma_1^{-1} = \sigma_1^{-2} = A_{12}^{-1},
$$
$$
c_{21} = b_{21} a_{21} =  \lambda_{32} (\lambda_{31} \lambda_{13}) \lambda_{23} = \sigma_2^{-1} \lambda_{21} \lambda_{12} \sigma_2^{-1} = \sigma_2^{-1} A_{12}^{-1} \sigma_2^{-1} = \sigma_2^{-2} \sigma_2 A_{12}^{-1} \sigma_2^{-1} = A_{23}^{-1} A_{13}^{-1},
$$
$$
c_{12} = b_{12} a_{12} =  \lambda_{21} (\lambda_{31} \lambda_{13}) \lambda_{12} = \sigma_1^{-1} \rho_1 \rho_2  \lambda_{21} \lambda_{12} \rho_2 \rho_1 \sigma_1^{-1} = \sigma_1^{-1} \rho_1  \lambda_{31} \lambda_{13}  \rho_1 \sigma_1^{-1} =
$$
$$
= \sigma_1^{-1}  \lambda_{32} \lambda_{23}  \sigma_1^{-1} =
\sigma_1^{-1} \sigma_2^{-1} \sigma_2^{-1} \sigma_1^{-1} =
\sigma_1^{-1} A_{23}^{-1} \sigma_1^{-1} = (\sigma_1^{-1} A_{23}^{-1} \sigma_1) A_{12}^{-1} = A_{13}^{-1} A_{12}^{-1}.
$$
And analogously,
$$
c_{31} = A_{34}^{-1} A_{24}^{-1} A_{14}^{-1},~~~c_{22} = A_{24}^{-1} A_{14}^{-1} A_{23}^{-1} A_{13}^{-1},~~~c_{13} = A_{14}^{-1} A_{13}^{-1} A_{12}^{-1}.
$$

In particular, we see that
$$
P_2 = T_1^c = \langle A_{12} \rangle,
$$
$$
P_3 = \langle T_1^c, T_2^c \rangle =  \langle c_{11},  c_{21}, c_{12} \rangle,
$$
$$
P_4 = \langle T_1^c, T_2^c, T_3^c \rangle =  \langle c_{11},  c_{21}, c_{12}, c_{31},  c_{22}, c_{13} \rangle.
$$

To find a set of defining relations, express  the old generators in the new one:
$$
A_{12} = c_{11}^{-1},~~A_{13} = c_{11} c_{12}^{-1},~~A_{23} = c_{12} c_{21}^{-1} c_{11}^{-1}.
$$
$$
A_{14} = c_{12} c_{13}^{-1},~~A_{24} = c_{13} c_{12}^{-1} c_{21} c_{22}^{-1},~~A_{34} = c_{22} c_{21}^{-1} c_{31}^{-1}.`
$$
Rewriting the set of defining relations of $P_4$ in the new generators we will find the set of defining relations.

Let us prove the formula for $c_{22}^{c_{12}^{-1}}$ and for $c_{22}^{c_{12}}$, assuming that all other formulas are true. Proofs for all others not difficult. Take the relation
$$
A_{13}^{-1} A_{24} A_{13} = [A_{14}^{-1}, A_{34}^{-1}] A_{24} [A_{34}^{-1}, A_{14}^{-1}].
$$
In the new generators this relation after cancellations has the form
$$
\left( c_{13} c_{12}^{-1} c_{21} c_{22}^{-1} \right)^{c_{11}} = c_{13}^{-1} c_{22} c_{21}^{-1} c_{31}^{-1} c_{13} c_{12}^{-1} c_{31} c_{21} c_{22}^{-1} c_{13} c_{31}^{-1} c_{13}^{-1} c_{31} c_{21} c_{22}^{-1} c_{13}.
$$
Using the formulas of conjugating by $c_{11}^{-1}$ we get
$$
c_{22} c_{13} c_{22}^{-1} c_{21} c_{12}^{-1} c_{22}^{-1} = c_{13}^{-1} c_{22} (c_{21}^{-1} c_{31}^{-1} c_{13}) c_{12}^{-1} c_{31} c_{21} c_{22}^{-1} c_{13} c_{31}^{-1} c_{13}^{-1} c_{31} (c_{21} c_{22}^{-1} c_{13}).
$$
Rewrite the term in the brackets in the form
$$
c_{21}^{-1} c_{31}^{-1} c_{13} = (c_{31}^{-1} c_{13})^{c_{21}} c_{21}^{-1} = c_{22}^{-1} c_{13} c_{22} c_{31}^{-1} c_{21}^{-1},
$$
$$
c_{21} c_{22}^{-1} c_{13} = (c_{22}^{-1} c_{13})^{c_{21}^{-1}} c_{21} = c_{31}^{-1} c_{13} c_{22}^{-1} c_{31} c_{21},
$$
we get
$$
(c_{13} c_{22}^{-1} c_{21}) c_{12}^{-1} c_{22}^{-1} = (c_{31}^{-1} c_{21}^{-1} c_{12}^{-1} c_{31}) c_{21} (c_{22}^{-1} c_{13} c_{31}^{-1} c_{22}^{-1} c_{31} c_{21}).
$$
Using the conjugation rules, rewrite the term in the brackets in the form
$$
c_{13} c_{22}^{-1} c_{21} = c_{21} (c_{13} c_{22}^{-1})^{c_{21}} = c_{21} c_{31} c_{22}^{-1}  c_{13} c_{31}^{-1},~~~
c_{31}^{-1} c_{21}^{-1} c_{12}^{-1} c_{31} = (c_{21}^{-1} c_{12}^{-1})^{c_{31}} = c_{21}^{-1} c_{12}^{-1},
$$
$$
c_{22}^{-1} c_{13} c_{31}^{-1} c_{22}^{-1} c_{31} c_{21}= c_{21} (c_{22}^{-1} c_{13} c_{31}^{-1} c_{22}^{-1} c_{31})^{c_{21}} = c_{21} c_{31} c_{22}^{-2} c_{13} c_{22} c_{31}^{-1} c_{22}^{-1},
$$
then
$$
c_{21} c_{31} c_{22}^{-1} c_{13} c_{31}^{-1} c_{12}^{-1} = c_{21}^{-1} c_{12}^{-1} c_{21}^{2} c_{31}  c_{22}^{-2} c_{13} c_{22} c_{31}^{-1}.
$$
Conjugating both sides by $c_{31}$ and using the fact that it commutes with $c_{12}$ and $c_{21}$, we get
\begin{equation} \label{conj}
c_{21} (c_{22}^{-1} c_{13}  c_{12}^{-1}) = c_{21}^{-1} c_{12}^{-1} c_{21}^{2} c_{22}^{-2} c_{13} c_{22}.
\end{equation}
Transforming the expression in the brackets, we get
$$
c_{21} c_{12}^{-1} c_{22}^{-c_{12}^{-1}}  c_{31}^{-1} c_{13} c_{31} = c_{21}^{-1} c_{12}^{-1} c_{21}^{2} c_{22}^{-2} c_{13} c_{22}.
$$
Multiply both sides to $c_{21}^{-2} c_{12} c_{21}$ on the left ant to $ c_{31}^{-1} c_{13}^{-1} c_{31}$ on the right we get

$$
[c_{21}^{2}, c_{12}^{-1}] c_{22}^{-c_{12}^{-1}}  = c_{22}^{-1} [c_{22}, c_{13}^{-1}] [c_{13}^{-1}, c_{31}].
$$
From this relation follows that

$$
c_{22}^{c_{12}^{-1}} = [c_{31}, c_{13}^{-1}] \, [c_{13}^{-1}, c_{22}] \, c_{22} \, [c_{21}^2, c_{12}^{-1}].
$$

\bigskip

To find conjugation formula $c_{22}^{c_{12}}$, we are using (\ref{conj})
$$
c_{21} c_{22}^{-1} c_{13}  c_{12}^{-1} = c_{21}^{-1} c_{12}^{-1} c_{21}^{2} (c_{22}^{-2} c_{13} c_{22}).
$$
Since
$$
(c_{22}^{-1} c_{13})^{c_{11}^{-1}} = c_{22}^{-2} c_{13} c_{22},
$$
then
$$
c_{21} (c_{22}^{-1} c_{13})  c_{12}^{-1} = c_{21}^{-1} c_{12}^{-1} c_{21}^{2} (c_{22}^{-1} c_{13})^{c_{11}^{-1}}.
$$
Multiply both sides to $c_{21}^{-2} c_{12} c_{21}$ to the left:
$$
[c_{21}^2, c_{12}^{-1}] (c_{22}^{-1} c_{13})^{c_{12}^{-1}} =  (c_{22}^{-1} c_{13})^{c_{11}^{-1}} \Leftrightarrow
(c_{22}^{-1} c_{13})^{c_{12}^{-1}} = [c_{12}^{-1}, c_{21}^{2}] (c_{22}^{-1} c_{13})^{c_{11}^{-1}}.
$$
Using the relation
$$
[c_{12}^{-1}, c_{21}^{2}] = [c_{12}^{-1}, c_{11}^{-2}]
$$
we have
$$
c_{22}^{-1} c_{13} c_{12}^{-1} =  c_{11}^{2} c_{12}^{-1} c_{11}^{-1} (c_{22}^{-1} c_{13}) c_{11}^{-1}.
$$
Multiply both sides to $c_{12}$ on the right
$$
c_{22}^{-1} c_{13} = c_{11}^{2} c_{12}^{-1} (c_{11}^{-1} c_{22}^{-1} c_{13}) c_{11}^{-1} c_{12}.
$$
Rewrite expression in the brackets
$$
c_{11}^{-1} c_{22}^{-1} c_{13} = (c_{22}^{-1} c_{13})^{c_{11}} c_{11}^{-1} = (c_{13} c_{22}^{-1}) c_{11}^{-1},
$$
then
$$
c_{22}^{-1} c_{13} = c_{11}^{2} c_{12}^{-1} (c_{13} c_{22}^{-1})  c_{11}^{-2} c_{12} \Leftrightarrow
c_{22}^{-1} c_{13} = c_{11}^{2}  c_{13}^{c_{12}} c_{22}^{-c_{12}} c_{12}^{-1}  c_{11}^{-2} c_{12}.
$$
Multiply both sides to $c_{13}^{-c_{12}} c_{11}^{-2}$ on the left and to $c_{12}^{-1}  c_{11}^{2} c_{12}$ on the right, we get
$$
c_{13}^{-c_{12}} (c_{11}^{-2} c_{22}^{-1} c_{13})  c_{12}^{-1} c_{11}^{2} c_{12} =  c_{22}^{-c_{12}}.
$$
Find the expression in the brackets
$$
c_{11}^{-2} c_{22}^{-1} c_{13} = (c_{22}^{-1} c_{13})^{c_{11}^{2}} c_{11}^{-2} = c_{22}^{-1} c_{13}^{c_{22}^{-2}} c_{11}^{-2}.
$$
Then
$$
c_{13}^{-c_{12}} c_{22}^{-1} c_{13}^{c_{22}^{-2}} [c_{11}^{2}, c_{12}] =  c_{22}^{-c_{12}}.
$$
Using the relations
$$
c_{13}^{-c_{12}} = c_{13}^{-c_{31}^{-1}},~~~[c_{11}^{2}, c_{12}] = [c_{21}^{-2}, c_{12}],
$$
we get
$$
c_{13}^{-c_{31}^{-1}} c_{22}^{-1} c_{13}^{c_{22}^{-2}} [c_{21}^{-2}, c_{12}] =  c_{22}^{-c_{12}}.
$$
from this relation follows the need relation.
\end{proof}

\subsection{Decomposition of $P_4$} In the paper \cite{CW} was proved that the Milnor simplicial group $F[S^1]$ is embedded into the simplicial group $AP_*$. The main problem in this theorem is the proof that groups $T_n^c$, $n = 2, 3, \ldots$, are free. To do it the authors used some Lie algebras. In this section we prove,  that $T_2^c$ and $T^c_3$ are free groups using group-theoretical methods.  Note, that $T_1^c$ is infinite cyclic.

From Proposition \ref{prop4.1} follows that $P_3$ has the following presentation
$$
P_3 = \langle c_{11},~~ c_{21},~~ c_{12}~||~
c_{21}^{c_{11}^{\varepsilon}} = c_{21},~~~c_{12}^{c_{11}^{\varepsilon}} = c_{12}^{c_{21}^{-\varepsilon}} \rangle.
$$
Hence, $PV_3 = T_2^c \leftthreetimes \mathbb{Z}$, where $T_2 = \langle c_{21}, c_{12} \rangle$ is a free group and $\mathbb{Z} = \langle c_{11} \rangle$.

To prove that $T_3^c$ is free,
define a homomorphism of $P_4$ onto free abelian group of rank 2:
$$
\varphi : P_4 \to \langle x, y ~||~ x y = y x \rangle = \mathbb{Z}^2,
$$
by the rule:
$$
\varphi(c_{11}) = x,~~\varphi(c_{21}) = y,~~\varphi(c_{12}) = e,~~\varphi(c_{31}) = e,~~\varphi(c_{22}) = e,~~\varphi(c_{13}) = e,
$$
where $e$ is the unit element of abelaian group.

Note that subgroup of $P_4$ that is generated by $c_{11}$ and $c_{21}$ is free abelian of rank 2. Hence, for the short exact sequens
$$
1 \to Ker (\varphi) \to P_4 \to \mathbb{Z}^2 \to 1,
$$
there exist a section $s : \mathbb{Z}^2 \to P_4$, $s(x) = a_{11}$, $s(y) = a_{21}$ and we have decomposition $P_4 = Ker (\varphi) \leftthreetimes \mathbb{Z}^2$ of $P_4$ into a semi-direct product.

Let us find a set of generators and defining relations for $Ker (\varphi)$. Put
$$
\Lambda = \{ c_{11}^k c_{21}^l ~|~k, l \in \mathbb{Z} \}
$$
is a set of coset representatives of $P_4$ by $s(\mathbb{Z}^2)$. Then $Ker (\varphi)$ is generated by elements
$$
c_{12}^{\lambda},~~c_{31}^{\lambda},~~c_{22}^{\lambda},~~c_{13}^{\lambda},~~\mbox{where}~~\lambda \in \Lambda.
$$
Using the following defining relations of $P_4$:
$$
c_{21}^{c_{11}^{\varepsilon}} = c_{21},~~~c_{12}^{c_{11}^{\varepsilon}} = c_{12}^{c_{21}^{-\varepsilon}},~~~c_{31}^{c_{11}^{\varepsilon}} = c_{31},~~~c_{22}^{c_{11}^{\varepsilon}} = c_{22},~~~c_{13}^{c_{11}^{\varepsilon}} = c_{13}^{c_{22}^{-\varepsilon}},
$$
$$
c_{31}^{c_{21}^{\varepsilon}} = c_{31},~~~c_{22}^{c_{21}^{\varepsilon}} = c_{22}^{c_{31}^{-\varepsilon}},~~~c_{13}^{c_{21}^{\varepsilon}} = c_{13}^{c_{22}^{\varepsilon} c_{31}^{-\varepsilon}},
$$
rewrite the generators of $Ker (\varphi)$ in the form
$$
c_{12}^{c_{11}^k c_{21}^l} = c_{12}^{c_{21}^{l-k}},~~~c_{31}^{c_{11}^k c_{21}^l} = c_{31},~~~c_{22}^{c_{11}^k c_{21}^l} = c_{22}^{c_{31}^{-l}},~~~c_{13}^{c_{11}^k c_{21}^l} = c_{13}^{c_{22}^{l-k} c_{31}^{-l}}.
$$
Hence, $Ker (\varphi)$ is generated by $c_{31}$, $c_{22}$, $c_{13}$ and infinite set $c_{12}^{c_{21}^{m}}$,~~$m \in \mathbb{Z}$. For simplicity we will denote $d_m = c_{12}^{c_{21}^{m}}$.

To find a set of defining relations of $Ker (\varphi)$, we take the last relations of $P_4$:
$$
c_{31}^{c_{12}} = c_{31},~~~c_{13}^{c_{12}} = c_{13}^{c_{31}^{-1}}.
$$
$$
c_{22}^{c_{12}^{-1}}  = c_{13}^{c_{31}} c_{13}^{-c_{22}}  c_{22} [c_{21}^2, c_{12}^{-1}].
$$
For simplicity, instead the last relation take relation
(\ref{conj}):
$$
c_{21} c_{22}^{-1} c_{13}  c_{12}^{-1} = c_{21}^{-1} c_{12}^{-1} c_{21}^{2} (c_{22}^{-2} c_{13} c_{22}),
$$
which is equivalent to the last one. Conjugating these relations by coset representatives $\lambda \in \Lambda$, we get a set of defining relations for $Ker (\varphi)$.

1) Conjugating the relation $c_{12}^{-1} c_{31} c_{12} = c_{31}$ by $c_{11}^k c_{21}^l$, we get
$$
c_{12}^{-c_{21}^{l-k}} c_{31} c_{12}^{c_{21}^{l-k}} = c_{31}.
$$
Put $m = l-k$, we get the set of relations
$$
d_m^{-1} c_{31} d_m = c_{31},~~m \in \mathbb{Z}.
$$

2) Conjugating the relation $c_{12}^{-1} c_{13} c_{12} = c_{31} c_{13} c_{31}^{-1}$ by $c_{11}^k c_{21}^l$, we get
$$
c_{12}^{-c_{21}^{l-k}} c_{13}^{c_{22}^{l-k} c_{31}^{-l}} c_{12}^{c_{21}^{l-k}} = c_{31} c_{13}^{c_{22}^{l-k} c_{31}^{-l}} c_{31}^{-1}.
$$
Conjugating this relation by $c_{31}^{l}$ and put $m = l-k$ we get the set of relations
$$
d_m^{-1} c_{13}^{c_{22}^{m} } d_m =  c_{13}^{c_{22}^{m} c_{31}^{-1}},~~m \in \mathbb{Z}.
$$

3) Conjugating the relation $c_{22}^{-1} c_{13}  c_{12}^{-1} =  c_{12}^{- c_{21}^{2}} c_{22}^{-2} c_{13} c_{22}$ by $c_{11}^k c_{21}^l$, we get
$$
c_{22}^{-c_{31}^{-l}} c_{13}^{c_{22}^{l-k} c_{31}^{-l}} c_{12}^{-c_{21}^{l-k}} =  c_{12}^{- c_{21}^{l-k+2}} \left( c_{22}^{-c_{31}^{-l}} \right)^{-2} c_{13}^{c_{22}^{l-k} c_{31}^{-l}} c_{22}^{c_{31}^{-l}}.
$$
Conjugating this relation by $c_{31}^{l}$ and put $m = l-k$ we get the set of relations
$$
c_{22}^{-1} c_{13}^{c_{22}^{l-k}} d_m^{-1} =  d_{m+2}^{-1} c_{22}^{2} c_{13}^{c_{22}^{m}} c_{22}.
$$
Hence, we prove

\begin{lem}
$Ker (\varphi)$ is generated by
$$
c_{31},~~c_{22},~~c_{13},~~d_m,~~m \in \mathbb{Z},
$$
and is defined by relations\\

1)~~~ $d_m^{-1} \, c_{31} \, d_m = c_{31},$\\

2)~~~ $ d_m^{-1} \, c_{13}^{c_{22}^{m}} \, d_m =  c_{13}^{c_{22}^{m} c_{31}^{-1}},$\\

3)~~~ $d_{m+2}  = c_{22}^{-1} \, c_{13}^{c_{22}^{m+1}} \, d_m  \, c_{13}^{-c_{22}^{m}} \, c_{22}$\\

for $m \in \mathbb{Z}$.

\end{lem}

From the set of relations 3) we express all generators $d_m$, $m \not= 0, 1$ as words in the generators
$$
d_0, d_1, c_{31}, c_{22}, c_{13}.
$$
If $m = 2 m_1 \geq 0$, then from 3) we have
$$
d_m = c_{22}^{-m_1} c_{13}^{c_{22}^{m_1}} c_{13}^{c_{22}^{m_1-1}} \ldots c_{13}^{c_{22}} \, d_0 \, c_{13}^{-1} \, c_{13}^{-c_{22}} \ldots c_{13}^{-c_{22}^{m_1-1}} c_{22}^{m_1}.
$$

If $m = 2 m_1 < 0$, then rewrite 3) in the form
$$
d_{m}  = c_{22} \, c_{13}^{-c_{22}^{m+2}} \, d_{m+2} \, c_{13}^{c_{22}^{m+1}} \, c_{22}^{-1}
$$
and by induction we get
$$
d_m = c_{22}^{-m_1} \, c_{13}^{-c_{22}^{-(m_1+1)}} \, c_{13}^{c_{22}^{-(m_1+2)}} \ldots c_{13}^{-c_{22}^{-1}} \, c_{13}^{-1} \, d_0 \,  \, c_{13}^{c_{22}^{-1}} \, c_{13}^{c_{22}^{-2}} \ldots c_{13}^{c_{22}^{-m_1}} c_{22}^{m_1}.
$$

Put these formulas into relations 2), we get

\begin{lem} \label{l4.3}
The set of relations 2) for the even indexes $m = 2 m_1$ is equivalent to the following set of relations:
$$
\left( c_{13}^{(c_{13} \, c_{22})^{m_1}} \right)^{d_0} = c_{13}^{c_{22}^{m} \, c_{31}^{-1} \, c_{22}^{-m} \,(c_{22} c_{13})^{m_1}}.
$$
\end{lem}

Now consider the odd indexes. If $m = 2 m_1 +1 > 0$, then from 3) we have
$$
d_m = c_{22}^{-m_1} c_{13}^{c_{22}^{m_1+1}} c_{13}^{c_{22}^{m_1}} \ldots c_{13}^{c_{22}^2} \, d_1 \,  c_{13}^{-c_{22}} \, c_{13}^{-c_{22}^2} \ldots c_{13}^{-c_{22}^{m_1}} c_{22}^{m_1}.
$$
If $m = 2 m_1 + 1 < 0$, then
$$
d_m = c_{22}^{-m_1} \, c_{13}^{-c_{22}^{m_1+2}} \, c_{13}^{-c_{22}^{m_1}+3} \ldots c_{13}^{-c_{22}^{-1}} \, c_{13}^{-1} \, c_{13}^{-c_{22}} \, d_1 \,  \, c_{13} \, c_{13}^{c_{22}^{-1}} \ldots c_{13}^{-c_{22}^{m_1+1}} c_{22}^{m_1}.
$$

Put these formulas into relations 2), we get

\begin{lem} \label{l4.4}
The set of relations 2) for the odd indexes $m = 2 m_1+1$ is equivalent to the following set of relations:
$$
\left( c_{13}^{(c_{22} \, c_{13})^{m_1-1} c_{22}^2} \right)^{d_1} = c_{13}^{c_{22}^{m} \, c_{31}^{-1} \, c_{22}^{-(m-1)} \,(c_{13} c_{22})^{m_1}}.
$$
\end{lem}

Considering the relations 1) and input the expressions for $d_m$ into these relations we get.

\begin{lem} \label{l4.5}
The set of relations 1) is equivalent to the union of the following sets of relations:

if $m = 2 m_1$ is even, then
$$
\left(  c_{31}^{ c_{22}^{-m} \, (c_{13} \, c_{22})^{m_1} } \right)^{d_0} = c_{31}^{c_{22}^{-m}   \,(c_{22} c_{13})^{m_1}};
$$

if $m = 2 m_1 + 1$ is odd, then
$$
\left(  c_{31}^{ c_{22}^{-(m+1)} \, (c_{22} \, c_{13})^{m_1} c_{22}^2} \right)^{d_1} = c_{31}^{c_{22}^{-(m-1)}   \,(c_{13} c_{22})^{m_1}}.
$$
\end{lem}

Hence we have proven

\begin{prop}
$Ker (\varphi)$ is generated by
$$
c_{31},~~c_{22},~~c_{13},~~d_0,~~d_1
$$
and is defined by relations from Lemmas \ref{l4.3} - \ref{l4.5}.
\end{prop}

Now we are going to prove that $Ker (\varphi)$ is two consequent HNN-extensions of the group $T_3^c = \langle c_{31}, c_{22}, c_{13} \rangle$. For this define subgroups $A_0, B_0, A_1, B_1$ of $G$. Let $m=2m_1$ is even number, then $A_0$ is generated by elements
$$
c_{13}^{ (c_{13} c_{22})^{m_1}},~~ c_{31}^{ c_{22}^{-m} (c_{13} c_{22})^{m_1}};
$$
$B_0$ is generated by elements
$$
c_{13}^{ c_{22}^m c_{31}^{-1} c_{22}^{-m} (c_{22} c_{13})^{m_1}},~~ c_{31}^{ c_{22}^{-m} (c_{22} c_{13})^{m_1}}.
$$
Define a map $\psi_0 : A_0 \to B_0$ on the generators:
$$
c_{13}^{ (c_{13} c_{22})^{m_1}} \to c_{13}^{ c_{22}^m c_{31}^{-1} c_{22}^{-m} (c_{22} c_{13})^{m_1}},~~~~ c_{31}^{ c_{22}^{-m} (c_{13} c_{22})^{m_1}} \to c_{31}^{ c_{22}^{-m} (c_{22} c_{13})^{m_1}}.
$$
From Lemmas \ref{l4.3},  \ref{l4.5} follows that $\psi_0$ is induced  conjugation by $d_0$ in $Ker (\varphi)$ and hence is an isomorphism.

Analogously, let $m=2m_1 + 1$ is odd number, then $A_1$ is generated by elements
$$
c_{13}^{ (c_{22} c_{13})^{(m_1-1)} c_{22}^2},~~ c_{31}^{ c_{22}^{-(m+1)} (c_{22} c_{13})^{m_1} c_{22}^2};
$$
$B_1$ is generated by elements
$$
c_{13}^{ c_{22}^m c_{31}^{-1} c_{22}^{-(m-1)} (c_{13} c_{22})^{m_1}},~~ c_{13}^{ c_{22}^{-(m-1)} (c_{13} c_{22})^{m_1}}.
$$
Define a map $\psi_1 : A_1 \to B_1$ on the generators:
$$
c_{13}^{ (c_{22} c_{13})^{(m_1-1)} c_{22}^2}
 \to c_{13}^{ c_{22}^m c_{31}^{-1} c_{22}^{-(m-1)} (c_{13} c_{22})^{m_1}},~~~~ c_{31}^{ c_{22}^{-(m+1)} (c_{22} c_{13})^{m_1} c_{22}^2} \to c_{13}^{ c_{22}^{-(m-1)} (c_{13} c_{22})^{m_1}}.
$$
From Lemmas \ref{l4.4},  \ref{l4.5} follows that $\psi_1$ is induced  conjugation by $d_1$ in $Ker (\varphi)$ and hence is an isomorphism. In these notations we have

\begin{thm} \label{t4.7}
$Ker (\varphi)$ is two consequent HNN-extensions with the base group $T_3^c$:
$$
Ker (\varphi) = \langle T_3^c, d_0, d_1~||~d_0^{-1} A_0 d_0 = B_0,  \psi_0; ~~d_1^{-1} A_1 d_1 = B_1, \psi_1 \rangle.
$$
\end{thm}

\begin{cor}
The group $T_3^c = \langle c_{31}, c_{22}, c_{13} \rangle$ is free of rank 3.
\end{cor}

\begin{proof}
The group $T_3^c$ is a subgroup of $Ker (\varphi)$. From Theorem \ref{t4.7} follows that all relations of $Ker (\varphi)$ are define $Ker (\varphi)$ as HNN-extensions, hence $T_3^c$ does not have defining relations.
\end{proof}

\section{Structure of $VP_3$ and $VP_4$} \label{s4}

The main purpose of this section find  sets of defining relations for $T_2$ and $T_3$. Note that $VP_3$ contains $T_2$ and has no commutativity relations,
$VP_4$ contains $T_3$ and has commutativity relations,
Relations of $T_n$ for $n > 3$ one can find using degeneracy maps $s_i$.

\subsection{The group $VP_3$}
In the generators
$$
\lambda_{12}, \lambda_{21}, \lambda_{13}, \lambda_{23}, \lambda_{31}, \lambda_{23},
$$
$VP_3$ is defined by the following 6 relations
$$
\lambda_{12} \lambda_{13} \lambda_{23} = \lambda_{23} \lambda_{13} \lambda_{12},~~~
\lambda_{21} \lambda_{23} \lambda_{13} = \lambda_{13} \lambda_{23} \lambda_{21},~~~
\lambda_{13} \lambda_{12} \lambda_{32} = \lambda_{32} \lambda_{12} \lambda_{13},
$$
$$
\lambda_{31} \lambda_{32} \lambda_{12} = \lambda_{12} \lambda_{32} \lambda_{31},~~~
\lambda_{23} \lambda_{21} \lambda_{31} = \lambda_{31} \lambda_{21} \lambda_{23},~~~
\lambda_{32} \lambda_{31} \lambda_{21} = \lambda_{21} \lambda_{31} \lambda_{32}.
$$

In the generators
$$
a_{11}, b_{11}, a_{21}, a_{12}, b_{21}, b_{12}
$$
$VP_3$ is defined by the following 6 relations
$$
[a_{21}, a_{12}] = 1,~~~b_{11} a_{11} a_{21} a_{11}^{-1} = a_{21} b_{11},~~~a_{12} b_{21} b_{12}^{-1} b_{11} = b_{21} b_{12}^{-1} b_{11} a_{11} a_{12} a_{11}^{-1},
$$
$$
b_{11}^{-1} b_{21} b_{11} a_{11} = a_{11} b_{21},~~~a_{11} a_{12}^{-1} a_{21} b_{12} = b_{11}^{-1} b_{12} b_{11} a_{11} a_{12}^{-1} a_{21},~~~
[b_{21}, b_{12}] = 1.
$$

In the paper \cite{BMVW} was found the following  decomposition of $VP_3$.

\begin{prop} \label{p4.1}
(\cite{BMVW}) The group $VP_3$ is generated by elements
$$
a_{11},~~ c_{11},~~ a_{21},~~ a_{12},~~
b_{21},~~ b_{12}
$$
and is defined by relations
$$
[a_{21}, a_{12}] = [b_{21}, b_{12}] = 1,
$$
$$
a_{21}^{c_{11}} = a_{21},~~~b_{21}^{c_{11}} = b_{21},~~~b_{12}^{c_{11}} =
b_{12}^{a_{21}^{-1} a_{12}},~~~ a_{12}^{c_{11}} = a_{12}^{b_{12}
a_{21}^{-1} a_{12} b_{21}^{-1}},
$$
i.~e. $VP_3 = \langle T_2, c_{11} \rangle * \langle a_{11} \rangle$,
$\langle T_2, c_{11} \rangle = T_2 \leftthreetimes \langle c_{11} \rangle.$
\end{prop}

In this proposition $c_{11}$ acts on $b_{12}$ and $a_{12}$ by different manner. Let us show that in fact these actions are equel. Indeed, since $a_{21}^{-1} a_{12} = a_{12} a_{21}^{-1}$, then
$$
a_{12}^{c_{11}} = a_{12}^{b_{12} a_{12} a_{21}^{-1} b_{21}^{-1}} \Leftrightarrow
a_{12}^{c_{11}} = a_{12}^{c_{12} c_{21}^{-1}}.
$$

Similarly, rewrite the conjugation rule
$$
b_{12}^{c_{11}} =
b_{12}^{a_{21}^{-1} a_{12}}
$$
in the form
$$
b_{12}^{c_{11}} =
b_{12}^{b_{12}  a_{12} a_{21}^{-1}}.
$$
Conjugating both  sides of this relation by $b_{21}^{-1}$ and using the fact that $c_{11}  b_{21}^{-1} = b_{21}^{-1} c_{11}$ and $b_{12}  b_{21} = b_{21} b_{12}$, we get
$$
b_{12}^{c_{11}} =
b_{12}^{c_{12} c_{21}^{-1}}.
$$
Hence, we have proven

\begin{cor} \label{c4.2}
The group $VP_3$ is generated by elements
$$
a_{11},~~ c_{11},~~ a_{21},~~ a_{12},~~
b_{21},~~ b_{12}
$$
and is defined by relations
$$
[a_{21}, a_{12}] = [b_{21}, b_{12}] = 1,
$$
$$
a_{21}^{c_{11}} = a_{21},~~~b_{21}^{c_{11}} = b_{21},~~~b_{12}^{c_{11}} =
b_{12}^{c_{12} c_{21}^{-1}},~~~ a_{12}^{c_{11}} = a_{12}^{c_{12} c_{21}^{-1}}.
$$
\end{cor}

Also, we can change  the generators $b_{ij}$ to the generators $c_{ij}$.

\begin{cor} \label{c4.3}
The group $VP_3$ is generated by elements
$$
a_{11},~~ c_{11},~~ a_{21},~~ a_{12},~~
c_{21},~~ c_{12}
$$
and is defined by relations
$$
[a_{21}, a_{12}] = [c_{21} a_{21}^{-1}, c_{12} a_{12}^{-1}] = 1,
$$
$$
a_{21}^{c_{11}} = a_{21},~~~c_{21}^{c_{11}} = c_{21},~~~a_{12}^{c_{11}} =
a_{12}^{c_{12} c_{21}^{-1}},~~~ c_{12}^{c_{11}} = c_{12}^{c_{21}^{-1}}.
$$
\end{cor}

\medskip

To find a set of defining relations of $T_2$ consider a homomorphism $\varphi : \langle T_2, c_{11} \rangle \to \langle c_{11} \rangle$ which sends all generators of $T_2$ to $1$ and sends $c_{11}$ to $c_{11}$. To find the kernel of this homomorphism,
we are using the  Reidemeister-Schreier method \cite[Section 2.3]{MKS}. The kernel is generated by elements
$$
S_{\lambda,a} = \lambda a \cdot (\overline{\lambda a})^{-1},\quad \lambda
\in \langle c_{11} \rangle,\quad
 a \in \{ a_{12}, a_{21}, b_{12}, b_{21}, c_{11} \}
$$
that are equal to
$$
c_{11}^{-k} a_{12} c_{11}^k,~~c_{11}^{-k} a_{21} c_{11}^k,~~c_{11}^{-k} b_{12} c_{11}^k,~~c_{11}^{-k} b_{21} c_{11}^k,~~k \in \mathbb{Z}.
$$

Defining relations of $Ker(\varphi)$ have the form
$$
c_{11}^{-k} \tau (r) c_{11}^k,~~k \in \mathbb{Z},
$$
where $r$ is a defining relation of the group $\langle T_2, c_{11} \rangle$ and $\tau$ is the rewriteble prosess (see \cite[Section 2.3]{MKS}). If $r$ runs through defining relation which are the conjugation rules, then we can use these defining relations to remove all generators of $Ker (\varphi)$ and keep only four generators:
$$
a_{12}, a_{21}, b_{12}, b_{21},
$$
It means  that the kernel is equal to $T_2$.
Hence, we have only relations
$$
[a_{21}, a_{12}]^{c_{11}^k} = [b_{21}, b_{12}]^{c_{11}^k} = 1, ~~~k \in \mathbb{Z}.
$$
We  proved

\begin{prop}
The group $T_2$ is generated by elements $a_{12}, a_{21}, b_{12}, b_{21}$ and is defined by relations
$$
[a_{21}, a_{12}]^{c_{11}^k} = [b_{21}, b_{12}]^{c_{11}^k} = 1, ~~~k \in \mathbb{Z}.
$$
\end{prop}

Using the conjugation rules in $VP_3$ one can prove

\begin{lem} \label{l4.12}
In $VP_3$ the following formulas hold
$$
a_{12}^{c_{11}^k} = a_{12}^{c_{12}^k c_{21}^{-k}},~~~b_{12}^{c_{11}^k} = b_{12}^{c_{12}^k c_{21}^{-k}},~~~k \in \mathbb{Z}.
$$
\end{lem}

Using these formulas we can give other description of $T_2$.

\begin{prop} \label{l4.13}
 $T_2$ is generated by elements
$$
a_{21}, a_{12}, b_{21}, b_{12}
$$
and is defined by the relations
$$
[a_{21}^{c_{21}^k}, a_{12}^{c_{12}^k}] = [b_{21}^{c_{21}^k}, b_{12}^{c_{12}^k}] = 1, ~~~k \in \mathbb{Z}.
$$
\end{prop}

\begin{proof}
As we know $T_2$ is defined by the relations
$$
[a_{21}, a_{12}]^{c_{11}^k} = [b_{21}, b_{12}]^{c_{11}^k} = 1.
$$
Using Lemma \ref{l4.12} and conjugation rules we can rewrite these relations in the form
$$
[a_{21}, a_{12}^{c_{12}^k c_{21}^{-k}}] = [b_{21}, b_{12}^{c_{12}^k c_{21}^{-k}}] = 1.
$$
Conjugating both sides of these relations by $c_{21}^{k}$, we get the need relations.
\end{proof}

\medskip

\subsection{$VP_4$ and its subgroup $T_3$} The group $VP_4$ is generated by elements
$$
\lambda_{12}, ~~\lambda_{21}, ~~\lambda_{13}, ~~\lambda_{23}, ~~\lambda_{31}, ~~\lambda_{32}, ~~
\lambda_{14}, ~~\lambda_{24}, ~~\lambda_{34}, ~~\lambda_{41}, ~~\lambda_{42}, ~~\lambda_{43}.
$$
On the over side, $VP_4 = \langle T_1, T_2, T_3 \rangle$, where
$$
T_1 = \langle a_{11}, b_{11} \rangle,~~T_2 = \langle a_{21}, a_{12}, b_{21}, b_{12} \rangle,~~
T_3 = \langle a_{31}, a_{22}, a_{13}, b_{31}, b_{22}, b_{13} \rangle.
$$
We have found expressions of the new generators $a_{ij}$ and $b_{ij}$ as words in standard generators of $VP_4$. Find expressions of the old generators
as words in the new generators:
$$
\lambda_{12} = a_{11}, ~~\lambda_{21} = b_{11}, ~~\lambda_{13}  = a_{12}  a_{11}^{-1}, ~~\lambda_{23}  = a_{11}  a_{12}^{-1} a_{21}, ~~
\lambda_{31} = b_{11}^{-1} b_{12},~~
\lambda_{32} = b_{21}  b_{12}^{-1} b_{11},
$$
$$
\lambda_{14}  = a_{13}  a_{12}^{-1},
~~\lambda_{24} = a_{12}  a_{13}^{-1} a_{22} a_{21}^{-1}, ~~\lambda_{34} = a_{21}  a_{22}^{-1} a_{31},
$$
$$
\lambda_{41}  = b_{12}^{-1} b_{13}, ~~\lambda_{42} = b_{21}^{-1} b_{22} b_{13}^{-1} b_{12}, ~~\lambda_{43} = b_{31}  b_{22}^{-1} b_{21}.
$$

To find a presentation of $VP_4$ in the new generators, we can act on $VP_3$ by degeneracy maps $s_1, s_2, s_3$. We will use a presentation of $VP_3$ from Corollary \ref{c4.2}. Then

1) The group $s_0(VP_3)$ is generated by elements
$$
a_{21},~~ c_{21},~~ a_{22},~~ a_{31},~~
b_{22},~~ b_{31}
$$
and is defined by relations
$$
[a_{31}, a_{22}] = [b_{31}, b_{22}] = 1,
$$
$$
a_{31}^{c_{21}} = a_{31},~~~b_{31}^{c_{21}} = b_{31},~~~a_{22}^{c_{21}} =
a_{22}^{c_{22} c_{31}^{-1}},~~~ b_{22}^{c_{21}} = b_{22}^{c_{22} c_{31}^{-1}}.
$$

2) The group $s_1(VP_3)$ is generated by elements
$$
a_{12},~~ c_{12},~~ a_{13},~~ a_{31},~~
b_{13},~~ b_{31}
$$
and is defined by relations
$$
[a_{31}, a_{13}] = [b_{31}, b_{13}] = 1,
$$
$$
a_{31}^{c_{12}} = a_{31},~~~b_{31}^{c_{12}} = b_{31},~~~a_{13}^{c_{12}} =
a_{13}^{c_{13} c_{31}^{-1}},~~~ b_{13}^{c_{12}} = b_{13}^{c_{13} c_{31}^{-1}}.
$$

3) The group $s_2(VP_3)$ is generated by elements
$$
a_{11},~~ c_{11},~~ a_{13},~~ a_{22},~~
b_{13},~~ b_{22}
$$
and is defined by relations
$$
[a_{22}, a_{13}] = [b_{22}, b_{13}] = 1,
$$
$$
a_{22}^{c_{11}} = a_{22},~~~b_{22}^{c_{11}} = b_{22},~~~a_{13}^{c_{11}} =
a_{13}^{c_{13} c_{22}^{-1}},~~~ b_{13}^{c_{11}} = b_{13}^{c_{13} c_{22}^{-1}}.
$$

The defining relations of the groups $VP_3$, $s_i(VP_3)$, $i = 0, 1, 2$, are not the full set of defining relations of $VP_4$. We need to add
the commutativity relations:
\begin{equation}
[\lambda_{34}^{*}, \lambda_{12}^{*}] = [\lambda_{24}^{*}, \lambda_{13}^{*}] = [\lambda_{14}^{*}, \lambda_{23}^{*}] = 1,
\label{r4.1}
\end{equation}
where $\lambda_{ij}^{*}$ is any element from the set $\{ \lambda_{ij}, \lambda_{ji} \}$.

To find defining relations of $T_3$ we need to understand that relations in $VP_4$ give relations in $T_3$. To do it we present $VP_4$ as HNN-extensions with some base group $G_4$ and stable letter $a_{11}$. Hence, the defining relations of $T_3$ came from defining relations of $G_4$.

We will analize the relations (\ref{r4.1}) and show that six from these relations are conjugation rules by $a_{11}$ and can be used in a presentation of $VP_4$ as HNN-extensions and other relations can be write as defining relations in $G_4$.

{\it Commutativity relations } $[\lambda_{34}^*, \lambda_{12}^*] = 1$. These relations have the form
$$
[\lambda_{34}, \lambda_{12}] = 1 \Leftrightarrow [a_{21} a_{22}^{-1} a_{31}, a_{11}] = 1,
$$
$$
[\lambda_{34}, \lambda_{21}] = 1 \Leftrightarrow [a_{21} a_{22}^{-1} a_{31}, b_{11}] = 1,
$$
$$
[\lambda_{43}, \lambda_{12}] = 1 \Leftrightarrow [b_{31} b_{22}^{-1} b_{21}, a_{11}] = 1,
$$
$$
[\lambda_{43}, \lambda_{21}] = 1 \Leftrightarrow [b_{31} b_{22}^{-1} b_{21}, b_{11}] = 1.
$$
Write the first and the third relations in the form
$$
\left( a_{21} a_{22}^{-1} a_{31} \right)^{a_{11}} = a_{21} a_{22}^{-1} a_{31},~~~\left( b_{31} b_{22}^{-1} b_{21} \right)^{a_{11}} = b_{31} b_{22}^{-1} b_{21}. $$
Then from the second and from the fourth relations follows
\begin{equation} \label{r4.3}
\left( a_{21} a_{22}^{-1} a_{31} \right)^{c_{11}} = a_{21} a_{22}^{-1} a_{31},~~~\left( b_{31} b_{22}^{-1} b_{21} \right)^{c_{11}} = b_{31} b_{22}^{-1} b_{21}.
\end{equation}
In $VP_3$ we have relations  $a_{21}^{c_{11}} = a_{21}$, $b_{21}^{c_{11}} = b_{21}$, and in $s_2(VP_3)$ we have relations $a_{22}^{c_{11}} = a_{22}$,  $b_{22}^{c_{11}} = b_{22}$. Hence, from (\ref{r4.3}) we get
$$
a_{31}^{c_{11}} =  a_{31},~~~ b_{31}^{c_{11}} = b_{31}.
$$
We proved

\begin{lem} \label{l4.6}
From the relations $[\lambda_{34}^*, \lambda_{12}^*] = 1$ in $VP_4$  follow formulas of  conjugation by $a_{11}$:
$$
\left( a_{21} a_{22}^{-1} a_{31} \right)^{a_{11}} = a_{21} a_{22}^{-1} a_{31},~~~\left( b_{31} b_{22}^{-1} b_{21} \right)^{a_{11}} = b_{31} b_{22}^{-1} b_{21}, $$
and  formulas of  conjugation by $c_{11}$:
$$
a_{31}^{c_{11}} =  a_{31},~~~ b_{31}^{c_{11}} = b_{31}.
$$
\end{lem}
{\it Commutativity relations } $[\lambda_{24}^*, \lambda_{13}^*] = 1$. These relations have the form
$$
[\lambda_{24}, \lambda_{13}] = 1 \Leftrightarrow [a_{12} a_{13}^{-1} a_{22} a_{21}^{-1}, a_{12} a_{11}^{-1}] = 1,
$$
$$
[\lambda_{24}, \lambda_{31}] = 1 \Leftrightarrow [a_{12} a_{13}^{-1} a_{22} a_{21}^{-1}, b_{11}^{-1} b_{12}] = 1,
$$
$$
[\lambda_{42}, \lambda_{13}] = 1 \Leftrightarrow [b_{21}^{-1} b_{22} b_{13}^{-1} b_{12}, a_{12} a_{11}^{-1}] = 1,
$$
$$
[\lambda_{42}, \lambda_{31}] = 1 \Leftrightarrow [b_{21}^{-1} b_{22} b_{13}^{-1} b_{12},  b_{11}^{-1} b_{12}] = 1.
$$

From the first relation we have the following conjugation formula by $a_{11}$:
\begin{equation} \label{r4.4}
\left( a_{12} a_{13}^{-1} a_{22} a_{21}^{-1} \right)^{a_{11}} = \left( a_{12} a_{13}^{-1} a_{22} a_{21}^{-1} \right)^{a_{12}}.
\end{equation}

The second relation
has the form
$$
\left( a_{12} a_{13}^{-1} a_{22} a_{21}^{-1} \right)^{b_{11}^{-1}} = \left( a_{12} a_{13}^{-1} a_{22} a_{21}^{-1} \right)^{b_{12}^{-1}}.
$$
Since $b_{ij}^{-1} = a_{ij} c_{ij}^{-1}$ we have
$$
\left( a_{12} a_{13}^{-1} a_{22} a_{21}^{-1} \right)^{a_{11} c_{11}^{-1}} = \left( a_{12} a_{13}^{-1} a_{22} a_{21}^{-1} \right)^{a_{12} c_{12}^{-1}}.
$$
Using (\ref{r4.4}) rewrite this relation in the form
$$
\left( a_{12} a_{13}^{-1} a_{22} a_{21}^{-1} \right)^{a_{12} c_{11}^{-1}} = \left( a_{13}^{-1} a_{22} a_{21}^{-1} a_{12}  \right)^{c_{12}^{-1}}.
$$
That is equivalent to the relation
\begin{equation}
\left( a_{13}^{-1} a_{22} \right)^{ c_{12}^{-1}} = \left( a_{13}^{-1} a_{22} \right)^{ c_{11}^{-1}}  \left( a_{21}^{-1} a_{12}\right)^{c_{11}^{-1}}
 \left( a_{12}^{-1} a_{21}\right)^{c_{12}^{-1}}.
\end{equation}

Similarly, from the third relation
\begin{equation} \label{r4.5}
\left( b_{21}^{-1} b_{22} b_{13}^{-1} b_{12} \right)^{a_{11}} = \left( b_{21}^{-1} b_{22} b_{13}^{-1} b_{12}\right)^{a_{12}}.
\end{equation}
It is a formula of conjugation by $a_{11}$.

The fourth relation has the form
$$
\left( b_{21}^{-1} b_{22} b_{13}^{-1} b_{12} \right)^{b_{11}^{-1}} = \left( b_{21}^{-1} b_{22} b_{13}^{-1} b_{12}\right)^{b_{12}^{-1}}.
$$
Using the equality  $b_{11}^{-1} = a_{11} c_{11}^{-1}$, rewrite the last relation in the form
$$
\left( b_{21}^{-1} b_{22} b_{13}^{-1} b_{12} \right)^{a_{11} c_{11}^{-1}} = b_{12} b_{21}^{-1} b_{22} b_{13}^{-1},
$$
and using (\ref{r4.5}) we get
$$
\left( b_{21}^{-1} b_{22} b_{13}^{-1} b_{12} \right)^{a_{12} c_{11}^{-1}} = b_{12} b_{21}^{-1} b_{22} b_{13}^{-1}.
$$
Since $a_{12} =  b_{12}^{-1} c_{12}$, we have
$$
\left( b_{12} b_{21}^{-1} b_{22} b_{13}^{-1} \right)^{c_{12}} = \left( b_{12} b_{21}^{-1} b_{22} b_{13}^{-1} \right)^{c_{11}}.
$$
This relation is equivalent to
\begin{equation}
\left( b_{22} b_{13}^{-1} \right)^{c_{12}} = \left( b_{21} b_{12}^{-1} \right)^{c_{12}} \left( b_{12} b_{21}^{-1} \right)^{c_{11}}
\left( b_{22} b_{13}^{-1} \right)^{c_{11}}.
\end{equation}

Hence, we have

\begin{lem} \label{l4.7}
From the relations $[\lambda_{24}^*, \lambda_{13}^*] = 1$ in $VP_4$  follow formulas of  conjugation by $a_{11}$:
$$
\left( a_{12} a_{13}^{-1} a_{22} a_{21}^{-1} \right)^{a_{11}} = \left( a_{12} a_{13}^{-1} a_{22} a_{21}^{-1} \right)^{a_{12}},~~~
\left( b_{21}^{-1} b_{22} b_{13}^{-1} b_{12} \right)^{a_{11}} = \left( b_{21}^{-1} b_{22} b_{13}^{-1} b_{12}\right)^{a_{12}},
$$
and  formulas of  conjugation by $c_{12}^{-1}$ and by $c_{12}$:
$$
\left( a_{13}^{-1} a_{22} \right)^{ c_{12}^{-1}} = \left( a_{13}^{-1} a_{22} \right)^{ c_{11}^{-1}}  \left( a_{21}^{-1} a_{12}\right)^{c_{11}^{-1}}
 \left( a_{12}^{-1} a_{21}\right)^{c_{12}^{-1}},
 $$
 $$
 \left( b_{22} b_{13}^{-1} \right)^{c_{12}} = \left( b_{21} b_{12}^{-1} \right)^{c_{12}} \left( b_{12} b_{21}^{-1} \right)^{c_{11}}
\left( b_{22} b_{13}^{-1} \right)^{c_{11}}.
$$
\end{lem}

Let us prove that we can simplify two last relations from this lemma.

\begin{lem} \label{l4.81}
In $VP_4$ the following relations hold

1) $\left( a_{13}^{-1} a_{22} \right)^{ c_{12}^{-1}} = \left( a_{13}^{-1} a_{22} \right)^{ c_{11}^{-1}}  [c_{21}, c_{12}^{-1}],$\\

2) $ \left( b_{22} b_{13}^{-1} \right)^{ c_{12}} = [c_{12}, c_{21}^{-1}] \left( b_{22} b_{13}^{-1} \right)^{ c_{11}}.$
\end{lem}

\begin{proof}
1) To prove the first relation, we need to prove the equality
$$
 \left( a_{21}^{-1} a_{12}\right)^{c_{11}^{-1}}
 \left( a_{12}^{-1} a_{21}\right)^{c_{12}^{-1}} = [c_{21}, c_{12}^{-1}].
 $$
 We have
$$
c_{11}^{-1} a_{12} a_{21}^{-1} c_{12}^{-1} = (c_{21}^{-1} c_{21}) c_{11}^{-1} a_{12} a_{21}^{-1} c_{12}^{-1},
$$
where we add unit element $1 = c_{21}^{-1} c_{21}$.

Since $c_{21} c_{11}^{-1} = c_{11}^{-1} c_{21}$ and $c_{21} a_{21}^{-1} \cdot a_{12} c_{12}^{-1} = a_{12} c_{12}^{-1} \cdot c_{21} a_{21}^{-1}$, the last expression has the form
$$
 c_{21}^{-1}  c_{11}^{-1} (c_{21}  a_{21}^{-1} \cdot a_{12} c_{12}^{-1}) =  c_{21}^{-1}  c_{11}^{-1} (a_{12} c_{12}^{-1} \cdot c_{21} a_{21}^{-1}) =
c_{21}^{-1}  (a_{12} c_{12}^{-1} \cdot c_{21} a_{21}^{-1})^{c_{11}} c_{11}^{-1} =
$$
$$
= c_{21}^{-1}  a_{12}^{c_{12} c_{21}^{-1}} c_{12}^{-c_{21}^{-1}} c_{21} a_{21}^{-1} c_{11}^{-1} = c_{12}^{-1}  a_{12} a_{21}^{-1} c_{11}^{-1}.
$$
Hence, we have the relation
$$
c_{11}^{-1} a_{12} a_{21}^{-1} c_{12}^{-1} = c_{12}^{-1}  a_{12} a_{21}^{-1} c_{11}^{-1}.
$$
From this relation follows
$$
a_{12} a_{21}^{-1} c_{12}^{-1} = (c_{12}^{-1}  a_{12} a_{21}^{-1})^{c_{11}^{-1}} \Leftrightarrow
(a_{12} a_{21}^{-1})^{c_{12}^{-1}} = c_{12}  c_{12}^{-c_{11}^{-1}} (a_{12} a_{21}^{-1})^{c_{11}^{-1}}\Leftrightarrow
$$
$$
\Leftrightarrow [c_{21}, c_{12}^{-1}]  = ( a_{21}^{-1} a_{12})^{c_{11}^{-1}} ( a_{12}^{-1} a_{21})^{c_{12}^{-1}}.
$$
From the last relation following the first relation in the lemma.

2) Let us prove the equality
$$
 \left( b_{21} b_{12}^{-1} \right)^{c_{12}}
 \left( b_{12} b_{21}^{-1} \right)^{c_{11}} = [c_{12}, c_{21}^{-1}].
 $$
Similarly to the previous case, we have
$$
c_{12} b_{12}^{-1} b_{21} c_{11} =  c_{12} b_{12}^{-1} b_{21} c_{11} (c_{21}^{-1} c_{21}) = (c_{12} b_{12}^{-1} \cdot b_{21} c_{21}^{-1}) c_{11} c_{21} =
$$
$$
= c_{11} ( b_{21} c_{21}^{-1} \cdot c_{12} b_{12}^{-1})^{c_{11}} c_{21} = c_{11} b_{21} c_{21}^{-1} \cdot c_{12}^{c_{21}^{-1}} b_{12}^{-c_{12} c_{21}^{-1}}) c_{21}) = c_{11} b_{21} b_{12}^{-1} c_{12}.
$$
Hence, we have found the relation
$$
c_{12} b_{12}^{-1} b_{21} c_{11} = c_{11} b_{21} b_{12}^{-1} c_{12}.
$$
From this relation
$$
c_{12}^{c_{11}} ( b_{12}^{-1} b_{21})^{c_{11}} =  b_{21} b_{12}^{-1} c_{12} \Leftrightarrow
c_{12}^{-1} c_{12}^{c_{21}^{-1}} ( b_{12}^{-1} b_{21})^{c_{11}} =  (b_{21} b_{12}^{-1})^{c_{12}}.
$$
This relation is equivalent to the need relation.
\end{proof}

\medskip

\begin{cor}
In $VP_4$ the following formulas hold

1) $a_{22}^{c_{12}^{-1}} = a_{13}^{c_{13}^{-1} c_{31}} a_{13}^{-c_{13}^{-1} c_{22}} a_{22} [c_{21}, c_{12}^{-1}],$\\

2) $b_{22}^{c_{12}} = [c_{12}, c_{21}^{-1}] b_{22} b_{13}^{-c_{13} c_{22}^{-1}} b_{13}^{c_{13} c_{31}^{-1}},$\\

3) $a_{22}^{c_{12}} = [c_{12}, c_{21}^{-1}]  a_{13}^{-c_{13} c_{22}^{-1}} a_{22}  a_{13}^{c_{13} c_{31}^{-1}},$\\

4) $b_{22}^{c_{12}^{-1}} = b_{13}^{c_{13}^{-1} c_{31}} b_{22} b_{13}^{-c_{13}^{-1} c_{22}}  [c_{21}, c_{12}^{-1}].$\\
\end{cor}

\begin{proof}
1) Let us prove the first formula. The prove of the second one is the same. Take the first relation in Lemma \ref{l4.81}:
$$
\left( a_{13}^{-1} a_{22} \right)^{ c_{12}^{-1}} = \left( a_{13}^{-1} a_{22} \right)^{ c_{11}^{-1}}  [c_{21}, c_{12}^{-1}].
$$
Using the conjugation formulas, we get
$$
 a_{13}^{-c_{13}^{-1} c_{31}} a_{22}^{c_{12}^{-1}} = a_{13}^{-c_{13}^{-1} c_{22}} a_{22}  [c_{21}, c_{12}^{-1}].
$$
From this relation we get the first formula:
$$
a_{22}^{c_{12}^{-1}} = a_{13}^{c_{13}^{-1} c_{31}} a_{13}^{-c_{13}^{-1} c_{22}} a_{22} [c_{21}, c_{12}^{-1}].
$$

3) Let us prove the third formula. The prove of the fourth one is the same. Take the first relation in Lemma \ref{l4.81}:
$$
\left( a_{13}^{-1} a_{22} \right)^{ c_{12}^{-1}} = \left( a_{13}^{-1} a_{22} \right)^{ c_{11}^{-1}}  [c_{21}, c_{12}^{-1}].
$$
Since
$$
[c_{21}, c_{12}^{-1}] = c_{11} c_{12} c_{11}^{-1} c_{12}^{-1},
$$
then the relation have the form
$$
a_{13}^{-1} a_{22} = c_{12}^{-1} c_{11} (a_{13}^{-1} a_{22}) c_{12} c_{11}^{-1}.
$$
Conjugating both sides by $c_{11}$ we get
$$
a_{13}^{-c_{11}} a_{22} = [c_{11}, c_{12}] a_{13}^{-c_{12}} a_{22}^{c_{12}}.
$$
Since $a_{13}$ and $a_{22}$ are commute, then
$$
a_{13}^{-c_{11}} a_{22} = [c_{11}, c_{12}] a_{22}^{c_{12}}  a_{13}^{-c_{12}}
$$
or
$$
a_{22}^{c_{12}} = [c_{12}, c_{11}] a_{13}^{-c_{11}} a_{22} a_{13}^{c_{12}}.
$$
Using the formulas
$$
[c_{12}, c_{11}] = [c_{12}, c_{21}^{-1}],~~~a_{13}^{c_{11}} = a_{13}^{c_{13} c_{22}^{-1}},~~~a_{13}^{c_{12}} = a_{13}^{c_{13} c_{31}^{-1}},
$$
we get the need relation.

\end{proof}

\medskip

{\it Commutativity relations } $[\lambda_{14}^*, \lambda_{23}^*] = 1$. These relations have the form
$$
[\lambda_{14}, \lambda_{23}] = 1 \Leftrightarrow [a_{13} a_{12}^{-1},  a_{11} a_{12}^{-1} a_{21}] = 1,
$$
$$
[\lambda_{14}, \lambda_{32}] = 1 \Leftrightarrow [a_{13} a_{12}^{-1},  b_{21} b_{12}^{-1} b_{11}] = 1,
$$
$$
[\lambda_{41}, \lambda_{23}] = 1 \Leftrightarrow [b_{12}^{-1} b_{13},  a_{11} a_{12}^{-1} a_{21}] = 1,
$$
$$
[\lambda_{41}, \lambda_{32}] = 1 \Leftrightarrow [b_{12}^{-1} b_{13},  b_{21} b_{12}^{-1} b_{11}] = 1.
$$

The first relation
gives the following conjugation rule by $a_{11}$:
\begin{equation} \label{r4.21}
\left( a_{13} a_{12}^{-1} \right)^{a_{11}} = \left( a_{13} a_{12}^{-1} \right)^{a_{21}^{-1} a_{12}}.
\end{equation}
The third  relation
gives the following conjugation rule
\begin{equation} \label{r4.22}
\left( b_{12}^{-1} b_{13} \right)^{a_{11}} = \left( b_{12}^{-1} b_{13}  \right)^{a_{21}^{-1} a_{12}}.
\end{equation}

Since $b_{ij} = c_{ij} a_{ij}^{-1}$, then
$$
b_{21} b_{12}^{-1} b_{11} = c_{21} a_{21}^{-1} a_{12} c_{12}^{-1} c_{11} a_{11}^{-1},
$$
and the second relation:
$$
[a_{13} a_{12}^{-1}, b_{21} b_{12}^{-1} b_{11}] = [a_{13} a_{12}^{-1}, c_{21} a_{21}^{-1} a_{12} c_{12}^{-1} c_{11} a_{11}^{-1}] = 1
$$
gives the following relation
$$
\left( a_{13} a_{12}^{-1} \right)^{a_{11} c_{11}^{-1} c_{12} a_{12}^{-1} a_{21} c_{21}^{-1}} =  a_{13} a_{12}^{-1}.
$$
Since $c_{12} a_{12}^{-1} \cdot a_{21} c_{21}^{-1} = a_{21} c_{21}^{-1} \cdot c_{12} a_{12}^{-1}$, then
$$
\left( a_{13} a_{12}^{-1} \right)^{a_{11} c_{11}^{-1} a_{21} c_{21}^{-1} \cdot c_{12}} =  (a_{13} a_{12}^{-1})^{ a_{12}}.
$$

Using  (\ref{r4.21})
 rewrite this  relation in the form
$$
\left( a_{13} a_{12}^{-1} \right)^{a_{21}^{-1} a_{12} c_{11}^{-1} a_{21} c_{21}^{-1} c_{12}} = \left( a_{13} a_{12}^{-1} \right)^{ a_{12}}.
$$
Since $a_{21}^{-1} a_{12} = a_{12} a_{21}^{-1}$, it is equivalent to
$$
\left( a_{12}^{-1} a_{13} \right)^{a_{21}^{-1} c_{11}^{-1} a_{21} c_{21}^{-1} c_{12}} =  a_{12}^{-1} a_{13}.
$$
Since $[a_{21}, c_{11}] = 1$, then
$$
\left( a_{12}^{-1} a_{13} \right)^{ c_{11}^{-1}  c_{21}^{-1} c_{12}} =  a_{12}^{-1} a_{13}.
$$
Using a conjugation formula by $c_{11}^{-1}$ we get
$$
\left( a_{12}^{-c_{12}^{-1} c_{21}} a_{13}^{c_{13}^{-1} c_{22}} \right)^{c_{21}^{-1} c_{12}} =  a_{12}^{-1} a_{13}.
$$
Hence
\begin{equation} \label{r3}
a_{13}^{c_{13}^{-1} c_{22}} = a_{13}^{c_{12}^{-1} c_{21}}.
\end{equation}

Similarly, the forth   relation
has the form
$$
\left( b_{12}^{-1} b_{13} \right)^{a_{11} c_{11}^{-1} c_{12} a_{12}^{-1} a_{21} c_{21}^{-1}} =  b_{12}^{-1} b_{13}.
$$

Using  (\ref{r4.22})
 rewrite  the forth relations in the form
$$
\left(  b_{12}^{-1} b_{13} \right)^{a_{21}^{-1} a_{12} c_{11}^{-1} c_{12} a_{12}^{-1} a_{21} c_{21}^{-1}} = b_{12}^{-1} b_{13}.
$$
Using the relation $c_{12} a_{12}^{-1} \cdot a_{21} c_{21}^{-1} = a_{21} c_{21}^{-1} \cdot c_{12} a_{12}^{-1}$ and $b_{12}^{-1} b_{13} = a_{12} c_{12}^{-1} c_{13} a_{13}^{-1} $  we can present this relation in the form
$$
\left( c_{12}^{-1} c_{13} a_{13}^{-1} a_{12} \right)^{a_{21}^{-1} c_{11}^{-1} a_{21} c_{21}^{-1} c_{12}} = c_{12}^{-1} c_{13} a_{13}^{-1} a_{12}.
$$

Since $[a_{21}, c_{11}] = 1$, we have
$$
\left( c_{12}^{-1} c_{13} a_{13}^{-1} a_{12} \right)^{ c_{11}^{-1}  c_{21}^{-1} c_{12}} = c_{12}^{-1} c_{13} a_{13}^{-1} a_{12}.
$$
Using the formulas of conjugating by $c_{11}^{-1}$ we get
$$
c_{13}^{c_{22}} a_{13}^{-c_{13}^{-1} c_{22}} = \left( c_{13} a_{13}^{-1} \right)^{c_{12}^{-1} c_{21}}.
$$
Using (\ref{r3}) we have
$$
c_{13}^{c_{22}}  =  c_{13}^{c_{12}^{-1} c_{21}}.
$$
Using a conjugation formula by $c_{12}^{-1}$
$$
c_{13}^{c_{22}}  = \left( c_{13}^{c_{31}} \right)^{c_{21}}.
$$
Conjugating both sides by $c_{21}^{-1}$
$$
c_{22}^{-c_{21}^{-1}} c_{13}^{c_{21}^{-1}} c_{22}^{c_{21}^{-1}} = c_{13}^{c_{31}}.
$$
Using the conjugation rules by $c_{21}^{-1}$
$$
c_{22}^{-c_{31}} c_{13}^{c_{21}^{-1}} c_{22}^{c_{31}} = c_{13}^{c_{31}},
$$
or
\begin{equation} \label{r41}
c_{13}^{c_{21}^{-1}} = c_{13}^{c_{22}^{-1} c_{31}}.
\end{equation}

\smallskip

Now come back to the relation (\ref{r3}) and write it in the form
$$
a_{13}^{c_{13}^{-1} c_{22} c_{21}^{-1}} = a_{13}^{c_{12}^{-1}}.
$$
Using the conjugation formulas, rewrite the left side, we arrive to relation
$$
(c_{22}^{-1})^{c_{21}^{-1}} c_{13}^{c_{21}^{-1}}  a_{13}^{c_{21}^{-1}} c_{13}^{-c_{21}^{-1}} c_{22}^{c_{21}^{-1}} = a_{13}^{c_{12}^{-1}}.
$$
Using the conjugation rules, we get
$$
c_{22}^{-c_{31}} c_{13}^{c_{22}^{-1} c_{31}}  a_{13}^{c_{21}^{-1}} c_{13}^{-c_{22}^{-1} c_{31}} c_{22}^{c_{31}} = a_{13}^{c_{13}^{-1} c_{31}}.
$$
It is equivalent to
$$
a_{13}^{c_{21}^{-1}} = c_{13}^{-c_{22}^{-1} c_{31}} c_{22}^{c_{31}} a_{13}^{c_{13}^{-1} c_{31}} c_{22}^{-c_{31}} c_{13}^{c_{22}^{-1} c_{31}}
$$
and after cancelations
$$
a_{13}^{c_{21}^{-1}} = a_{13}^{c_{22}^{-1} c_{31}}.
$$

Since $b_{13} = c_{13} a_{13}^{-1}$, then using the last relation and relation (\ref{r41}), we get
$$
b_{13}^{c_{21}^{-1}} = b_{13}^{c_{22}^{-1} c_{31}}.
$$
Hence, we proved

\begin{lem}
The commutativity relations $[\lambda_{14}^*, \lambda_{23}^*] = 1$ in $VP_4$ give the following conjugation formulas by $a_{11}$:
$$
\left( a_{13} a_{12}^{-1} \right)^{a_{11}} = \left( a_{13} a_{12}^{-1} \right)^{ a_{21}^{-1} a_{12}},~~~
\left( b_{12}^{-1} b_{13} \right)^{a_{11}} = \left( b_{12}^{-1} b_{13} \right)^{a_{21}^{-1} a_{12}},
$$
and the conjugation formulas by $c_{21}^{-1}$:
$$
a_{13}^{c_{21}^{-1}} = a_{13}^{c_{22}^{-1} c_{31}},~~~
b_{13}^{c_{21}^{-1}} = b_{13}^{c_{22}^{-1} c_{31}}.
$$
\end{lem}

\subsection{$VP_4$ as HNN-extension} From the relations of commutativity in $VP_4$  we got the following conjugation formulas by element $a_{11}$:
$$
\left( a_{21} a_{22}^{-1} a_{31} \right)^{a_{11}} = a_{21} a_{22}^{-1} a_{31},~~~\left( b_{31} b_{22}^{-1} b_{21} \right)^{a_{11}} = b_{31} b_{22}^{-1} b_{21},
$$
$$
\left( a_{12} a_{13}^{-1} a_{22} a_{21}^{-1} \right)^{a_{11}} = \left( a_{12} a_{13}^{-1} a_{22} a_{21}^{-1} \right)^{a_{12}},~~~
\left( b_{21}^{-1} b_{22} b_{13}^{-1} b_{12} \right)^{a_{11}} = \left( b_{21}^{-1} b_{22} b_{13}^{-1} b_{12}\right)^{a_{12}},
$$
$$
\left( a_{13} a_{12}^{-1} \right)^{a_{11}} = \left( a_{13} a_{12}^{-1} \right)^{ a_{21}^{-1} a_{12}},~~~
\left( b_{12}^{-1} b_{13} \right)^{a_{11}} = \left( b_{12}^{-1} b_{13} \right)^{a_{21}^{-1} a_{12}}.
$$

Denote
$$
A_a = \langle a_{21} a_{22}^{-1} a_{31},~~a_{12} a_{13}^{-1} a_{22} a_{21}^{-1},~~ a_{13} a_{12}^{-1} \rangle.
$$
We see that
$$
A_a = \langle \lambda_{34},~~\lambda_{24},~~ \lambda_{14} \rangle.
$$

Denote
$$
A_b = \langle b_{31} b_{22}^{-1} b_{21},~~b_{21}^{-1} b_{22} b_{13}^{-1} b_{12},~~ b_{12}^{-1} b_{13} \rangle.
$$
We see that
$$
A_b = \langle \lambda_{43},~~\lambda_{42},~~ \lambda_{41} \rangle.
$$

Also denote
$$
B_a = \langle a_{21} a_{22}^{-1} a_{31},~~\left( a_{12} a_{13}^{-1} a_{22} a_{21}^{-1} \right)^{a_{12}},~~ \left( a_{13} a_{12}^{-1} \right)^{ a_{21}^{-1} a_{12}} \rangle = \langle \lambda_{34},~~\lambda_{24}^{\lambda_{12}},~~ \lambda_{14}^{\lambda_{12}} \rangle
$$
and
$$
B_b = \langle b_{31} b_{22}^{-1} b_{21},~~\left( b_{21}^{-1} b_{22} b_{13}^{-1} b_{12}\right)^{a_{12}},~~ \left( b_{12}^{-1} b_{13} \right)^{a_{21}^{-1} a_{12}} \rangle = \langle  \lambda_{43},~~\lambda_{42}^{\lambda_{12}},~~ \lambda_{41}^{\lambda_{12}}\rangle.
$$
We see that $B_a = A_a^{a_{11}}$, $B_b = A_b^{a_{11}}$. Put $A = \langle A_a, A_b \rangle$, $B = \langle B_a, B_b \rangle$. Since $B$ is conjugate with $A$, then $A$ is isomorphic to $B$ and we get

\begin{thm} \label{t4.11}
$VP_4$ is the HNN-extension with the base group
$$
G_4 = \langle c_{11}, a_{21}, a_{12}, c_{21}, c_{12}, a_{31}, a_{22}, a_{13}, b_{31}, b_{22}, b_{13} \rangle
$$
associated subgroups $A$ and $B$,  stable letter $a_{11}$. $G_4$ is defined by the following relations
(here $\varepsilon = \pm 1$):

1) conjugations by  $c_{11}^{\varepsilon}$

$$
a_{21}^{c_{11}^{\varepsilon}} = a_{21},~~~a_{12}^{c_{11}^{\varepsilon}} = a_{12}^{c_{12}^{\varepsilon} c_{21}^{-\varepsilon}},~~~c_{21}^{c_{11}^{\varepsilon}} = c_{21},~~~c_{12}^{c_{11}^{\varepsilon}} = c_{12}^{c_{21}^{-\varepsilon}},
$$

$$
a_{31}^{c_{11}^{\varepsilon}} = a_{31},~~~a_{22}^{c_{11}^{\varepsilon}} = a_{22},~~~
a_{13}^{c_{11}^{\varepsilon}} = a_{13}^{c_{13}^{\varepsilon} c_{22}^{-\varepsilon}},~~~b_{31}^{c_{11}^{\varepsilon}} = b_{31},~~~
b_{22}^{c_{11}^{\varepsilon}} = b_{22},~~~
b_{13}^{c_{11}^{\varepsilon}} = b_{13}^{c_{13}^{\varepsilon} c_{22}^{-\varepsilon}},
$$

2) conjugations by  $c_{21}^{\varepsilon}$

$$
a_{31}^{c_{21}^{\varepsilon}} = a_{31},~~~a_{22}^{c_{21}^{\varepsilon}} = a_{22}^{c_{22}^{\varepsilon} c_{31}^{-\varepsilon}},~~~
a_{13}^{c_{21}^{\varepsilon}} = a_{13}^{c_{22}^{\varepsilon} c_{31}^{-\varepsilon}},~~~b_{31}^{c_{21}^{\varepsilon}} = b_{31},~~~
b_{22}^{c_{21}^{\varepsilon}} = b_{22}^{c_{22}^{\varepsilon} c_{31}^{-\varepsilon}},~~~
b_{13}^{c_{21}^{\varepsilon}} = b_{13}^{c_{22}^{\varepsilon} c_{31}^{-\varepsilon}},
$$

3) conjugations by  $c_{12}^{\varepsilon}$

$$
a_{31}^{c_{12}^{\varepsilon}} = a_{31},~~~
a_{13}^{c_{12}^{\varepsilon}} = a_{13}^{c_{13}^{\varepsilon} c_{31}^{-\varepsilon}},~~~b_{31}^{c_{12}^{\varepsilon}} = b_{31},~~~
b_{13}^{c_{12}^{\varepsilon}} = b_{13}^{c_{13}^{\varepsilon} c_{31}^{-\varepsilon}},
$$
$$
a_{22}^{c_{12}^{-1}} = a_{13}^{c_{13}^{-1} c_{31}} a_{13}^{-c_{13}^{-1} c_{22}} a_{22} [c_{21}, c_{12}^{-1}],~~
a_{22}^{c_{12}} = [c_{12}, c_{21}^{-1}]  a_{13}^{-c_{13} c_{22}^{-1}} a_{22}  a_{13}^{c_{13} c_{31}^{-1}},
$$
$$
b_{22}^{c_{12}^{-1}} = b_{13}^{c_{13}^{-1} c_{31}} b_{22} b_{13}^{-c_{13}^{-1} c_{22}}  [c_{21}, c_{12}^{-1}],~~
b_{22}^{c_{12}} = [c_{12}, c_{21}^{-1}] b_{22} b_{13}^{-c_{13} c_{22}^{-1}} b_{13}^{c_{13} c_{31}^{-1}}.
$$

4) commutativity relations
$$
[a_{21}, a_{12}] = [a_{31}, a_{22}] = [a_{31}, a_{13}] = [a_{22}, a_{13}] = 1,
$$
$$
[c_{21} a_{21}^{-1}, c_{12} a_{21}^{-1}] = [b_{31}, b_{22}] = [b_{31}, b_{13}] = [b_{22}, b_{13}] = 1.
$$

\end{thm}

Hence, to find defining relations of $T_3$ we need to study $G_4$.

Define the following subgroup of $G_4$:
$$
Q = \langle a_{21}, a_{12}, c_{21}, c_{12}, a_{31}, a_{22}, a_{13}, b_{31}, b_{22}, b_{13} \rangle.
$$
From relations 1) of Theorem \ref{t4.11} follows that $Q$ is normal in $G_4$ and is the kernel of the homomorphism
$$
G_4 \longrightarrow \langle c_{11} \rangle
$$
which sends $c_{11}$ to $c_{11}$ and sends all other generators to 1.
Similarly to the case $VP_3$ one can  see that $Q$  is defined by relations which come from
 relations 2) -- 4) of  Theorem \ref{t4.11} by conjugation $c_{11}^k$, $k \in \mathbb{Z}$. Using the defining relations of $G_4$ one can prove that all conjugations of relations 2) -- 3) are equivalent to relations 2) -- 3). Hence,

 \begin{lem} \label{l5.13}
 The group $Q$ is defined by relations 2) -- 3) of  Theorem \ref{t4.11} and relations
$$
[a_{21}, a_{12}]^{c_{11}^k} = [a_{31}, a_{22}]^{c_{11}^k} = [a_{31}, a_{13}]^{c_{11}^k} = [a_{22}, a_{13}]^{c_{11}^k} = 1,
$$
$$
[c_{21} a_{21}^{-1}, c_{12} a_{21}^{-1}]^{c_{11}^k} = [b_{31}, b_{22}]^{c_{11}^k} = [b_{31}, b_{13}]^{c_{11}^k} = [b_{22}, b_{13}]^{c_{11}^k} = 1,
$$
 that can be written in the form
$$
[a_{21}, a_{12}^{c_{12}^k c_{21}^{-k}}] = [a_{31}, a_{22}] = [a_{31}, a_{13}^{c_{13}^k c_{22}^{-k}}] = [a_{22}, a_{13}^{c_{13}^k c_{22}^{-k}}] = 1,
$$
$$
[c_{21} a_{21}^{-1}, c_{12}^{c_{21}^{-k}} a_{21}^{-c_{12}^k c_{21}^{-k}}] = [b_{31}, b_{22}] = [b_{31}, b_{13}^{c_{13}^k c_{22}^{-k}}] = [b_{22}, b_{13}^{c_{13}^k c_{22}^{-k}}] = 1,
$$
for all integer  numbers $k$.
\end{lem}

Now we can prove the main result of the present paper.

\begin{thm}\label{T_3}
The group
$$
T_3 = \langle a_{31},~~a_{22},~~a_{13},~~b_{31},~~b_{22},~~b_{13} \rangle
$$
is defined by relations
$$
[a_{31}, a_{22}^{c_{22}^{m} c_{31}^{-m}}] = [a_{31}, a_{13}^{c_{13}^{k} c_{22}^{m-k} c_{31}^{-m}}] = [a_{22}^{c_{22}^m c_{31}^{-m}}, a_{13}^{c_{13}^k c_{22}^{m-k} c_{31}^{-m}}] = 1,
$$
$$
[b_{31}, b_{22}^{c_{22}^{m} c_{31}^{-m}}] = [b_{31}, b_{13}^{c_{13}^{k} c_{22}^{m-k} c_{31}^{-m}}] = [b_{22}^{c_{22}^m c_{31}^{-m}}, b_{13}^{c_{13}^k c_{22}^{m-k} c_{31}^{-m}}] = 1.
$$
where $k, m \in \mathbb{Z}$.
\end{thm}

\begin{proof}

In Lemma \ref{l5.13} we  have found a set of defining relations for $Q$. From this set follows that $Q$ is a free product of subgroups
$Q_1 = \langle c_{21}, c_{12},  a_{21}, a_{12}  \rangle$ and $Q_2 = \langle c_{21}, c_{12}, a_{31}, a_{22}, a_{13}, b_{31}, b_{22}, b_{13} \rangle$ with amalgamated
subgroup $Q_c = \langle c_{21}, c_{12} \rangle$. Hence, we have a set of defining relations for $Q_2$: it includes  relations of $Q$, that contains only generators of $Q_2$.

Now consider a homomorphism
$$
\varphi : Q_2 \longrightarrow \langle c_{21} \rangle,
$$
that is defined by the formulas
$$
\varphi (c_{21}) = c_{21},~~ \varphi (c_{12}) = \varphi (a_{31}) = \varphi (a_{22}) = \varphi (a_{13}) = \varphi (b_{31}) = \varphi (b_{22}) = \varphi (b_{13}) = 1.
$$
To find a presentation of $Ker(\varphi)$ take the set of coset representatives of this kernel in $Q_2$:
$$
\Lambda = \{ c_{21}^m ~|~ m \in \mathbb{Z} \}.
$$

Then  $Ker(\varphi)$ is generated by elements
$$
c_{12}^{c_{21}^m},~~a_{31}^{c_{21}^m},~~a_{22}^{c_{21}^m},~~a_{13}^{c_{21}^m},~~
b_{31}^{c_{21}^m},~~b_{22}^{c_{21}^m},~~b_{13}^{c_{21}^m}.
$$
Let us denote
$$
d_m = c_{12}^{c_{21}^m},~~~m \in \mathbb{Z}.
$$
Using the conjugations formulas by $c_{21}$ from Theorem \ref{t4.11}, we get
$$
a_{31}^{c_{21}^m} = a_{31},~~a_{22}^{c_{21}^m} = a_{22}^{c_{22}^m c_{31}^{-m}},~~~a_{13}^{c_{21}^m} = a_{13}^{c_{22}^m c_{31}^{-m}},
$$
$$
b_{31}^{c_{21}^m} = b_{31},~~b_{22}^{c_{21}^m} = b_{22}^{c_{22}^m c_{31}^{-m}},~~~b_{13}^{c_{21}^m} = b_{13}^{c_{22}^m c_{31}^{-m}}.
$$
Hence $Ker(\varphi)$ is generated by elements
$$
d_m, m \in \mathbb{Z}, ~a_{31}, a_{22}, a_{13}, b_{31}, b_{22}, b_{13}.
$$

To find a set of defining relations for $Ker(\varphi)$ we have to take the following relations in $Q$:
$$
a_{31}^{c_{12}} = a_{31},~~~
a_{13}^{c_{12}} = a_{13}^{c_{13} c_{31}^{-1}},~~~b_{31}^{c_{12}} = b_{31},~~~
b_{13}^{c_{12}} = b_{13}^{c_{13} c_{31}^{-1}},
$$
$$
a_{22}^{c_{12}^{-1}} = a_{13}^{c_{13}^{-1} c_{31}} a_{13}^{-c_{13}^{-1} c_{22}} a_{22} [c_{21}, c_{12}^{-1}],~~
b_{22}^{c_{12}^{-1}} = b_{13}^{c_{13}^{-1} c_{31}} b_{22} b_{13}^{-c_{13}^{-1} c_{22}}  [c_{21}, c_{12}^{-1}],
$$
$$
[a_{21}, a_{12}^{c_{12}^k c_{21}^{-k}}] = [a_{31}, a_{22}] = [a_{31}, a_{13}^{c_{13}^k c_{22}^{-k}}] = [a_{22}, a_{13}^{c_{13}^k c_{22}^{-k}}] = 1,
$$
$$
[c_{21} a_{21}^{-1}, c_{12}^{c_{21}^{-k}} a_{21}^{-c_{12}^k c_{21}^{-k}}] = [b_{31}, b_{22}] = [b_{31}, b_{13}^{c_{13}^k c_{22}^{-k}}] = [b_{22}, b_{13}^{c_{13}^k c_{22}^{-k}}] = 1,
$$
and conjugate them by $c_{21}^m$.

At first consider the relations
$$
a_{22}^{c_{12}^{-1}} = a_{13}^{c_{13}^{-1} c_{31}} a_{13}^{-c_{13}^{-1} c_{22}} a_{22} [c_{21}, c_{12}^{-1}],~~
b_{22}^{c_{12}^{-1}} = b_{13}^{c_{13}^{-1} c_{31}} b_{22} b_{13}^{-c_{13}^{-1} c_{22}}  [c_{21}, c_{12}^{-1}].
$$
These relations are equivalent to the following relations from Lemma \ref{l4.81}
\begin{equation} \label{r5.11}
\left( a_{13}^{-1} a_{22} \right)^{ c_{12}^{-1}} = \left( a_{13}^{-1} a_{22} \right)^{ c_{11}^{-1}}  [c_{21}, c_{12}^{-1}],
\end{equation}
\begin{equation} \label{r5.12}
\left( b_{22} b_{13}^{-1} \right)^{ c_{12}} = [c_{12}, c_{21}^{-1}] \left( b_{22} b_{13}^{-1} \right)^{ c_{11}}.
\end{equation}
Using the formulas of conjugations by $c_{11}^{-1}$, rewrite the relation (\ref{r5.11}) in the form
$$
d_0 \, a_{13}^{-1} \, a_{22} = a_{13}^{-c_{13}^{-1} c_{22}} \, a_{22} \, d_1.
$$
Conjugated it by $c_{21}^m$ we get
$$
d_m \, a_{13}^{-c_{22}^{m} c_{31}^{-m}}  a_{22}^{c_{22}^{m} c_{31}^{-m}} = a_{13}^{-c_{13}^{-1} c_{22}^{m} c_{22} c_{31}^{-m}} a_{22}^{c_{22}^{m} c_{31}^{-m}} d_{m+1}.
$$
Conjugated both sides of this relation by $c_{31}^{m}$ we  have
$$
d_m \, \left( a_{13}^{-1}  a_{22}\right)^{c_{22}^{m}} = a_{13}^{-c_{13}^{-1} c_{22}^{m+1}} a_{22}^{c_{22}^{m}} d_{m+1}.
$$
From these relations we have
$$
d_{m+1} =  \left( a_{22}^{-1}  a_{13}^{c_{13}^{-1} c_{22}} \right)^{c_{22}^{m}} \, d_m  \left( a_{13}^{-1}  a_{22} \right)^{c_{22}^{m}}~~\mbox{for}~m \geq 0,
$$
$$
d_{m} =  \left( a_{13}^{-c_{13}^{-1} c_{22}} a_{22} \right)^{c_{22}^{m}} \, d_{m+1}  \left( a_{22}^{-1}  a_{13} \right)^{c_{22}^{m}}~~\mbox{for}~m < 0.
$$
Analogously, from (\ref{r5.12}) we get the following formulas
$$
d_{m} =  \left( b_{13}  b_{22}^{-1}  \right)^{c_{22}^{m}} \, d_{m-1}  \left(  b_{22}  b_{13}^{-c_{13} c_{22}^{-1}} \right)^{c_{22}^{m}}~~\mbox{for}~m \geq 1,
$$
$$
d_{m-1} =  \left( b_{22} b_{13}^{-1}  \right)^{c_{22}^{m}} \, d_{m}  \left( b_{13}^{c_{13} c_{22}^{-1}} b_{22}^{-1} \right)^{c_{22}^{m}}~~\mbox{for}~m < 1.
$$

For further calculations introduce the notations
$$
A_1^{(l)} = \left( a_{22}^{-1}  \, a_{13}^{c_{13}^{-1} c_{22}} \right)^{c_{22}^{l}},~~~A_2^{(l)} = \left( a_{13}^{-1}  \, a_{22} \right)^{c_{22}^{l}},~~\overline{A}_i^{(l)} = \left( A_i^{(l)} \right)^{-1},
$$
$$
B_1^{(l)} = \left( b_{22}  \, b_{13}^{-c_{13} c_{22}^{-1}} \right)^{c_{22}^{l}},~~~B_2^{(l)} = \left( b_{13}  \, b_{22}^{-1} \right)^{c_{22}^{l}},~~\overline{B}_i^{(l)} = \left( B_i^{(l)} \right)^{-1},
$$
for all integers $l$.

Using these notations we express  $d_m$, $m \not= 0$ as  words, which depend only on $d_0^{\pm 1}$ and other generators of $Ker(\varphi)$. Using induction on $m$ we get:

for $m \geq 1$

$$
d_m = A_1^{(m)} \, A_1^{(m-1)} \ldots A_1^{(0)} \, d_0 \,  A_2^{(0)} \, A_2^{(1)} \ldots A_2^{(m)},
$$
$$
d_m = B_2^{(m)} \, B_2^{(m-1)} \ldots B_2^{(1)}\,  d_0 \, B_1^{(1)} \, B_1^{(2)} \ldots B_1^{(m)},
$$

for $m \leq -1$

$$
d_m = \overline{A}_1^{(m)} \overline{A}_1^{(m+1)} \ldots \overline{A}_1^{(-1)} \, d_0 \, \overline{A}_2^{(-1)} \overline{A}_2^{(-2)}  \ldots \overline{A}_2^{(m)},
$$
$$
d_m = \overline{B}_2^{(m+1)} \overline{B}_2^{(m+2)} \ldots \overline{B}_2^{(0)} \, d_0 \, \overline{B}_1^{(0)} \overline{B}_1^{(-1)}  \ldots \overline{B}_1^{(m+1)}.
$$

We see that left sides of these relations are equal, then equality of the right sides  gives  relations:

for $m \geq 1$

$$
d_0^{-1} \left(  \overline{A}_1^{(0)} \, \overline{A}_1^{(1)} \ldots \overline{A}_1^{(m)} \cdot  B_2^{(m)} \, B_2^{(m-1)} \ldots B_2^{(1)} \right) d_0 =
$$
$$
=A_2^{(0)} \, A_2^{(1)} \ldots A_2^{(m)} \cdot \overline{B}_1^{(m)} \overline{B}_1^{(m-1)} \ldots \overline{B}_1^{(1)},
$$

for $m \leq -1$

$$
d_0^{-1} \left( A_1^{(-1)} \, A_1^{(-2)} \ldots A_1^{(m)} \cdot \overline{ B}_2^{(m+1)} \, \overline{B}_2^{(m+2)} \ldots \overline{B}_2^{(0)} \right) d_0 =
$$
$$
=\overline{A}_2^{(-1)} \, \overline{A}_2^{(-2)} \ldots \overline{A}_2^{(m)} \cdot \overline{B}_1^{(m+1)} \overline{B}_1^{(m+2)} \ldots \overline{B}_1^{(0)}.
$$

Let us consider other relations.

1) Take the relation $c_{12}^{-1} a_{31} c_{12} = a_{31}$. Conjugating it by $c_{21}^m$ we get $d_{m}^{-1} a_{31} d_{m} = a_{31}$. Put instead $d_m$ its expressions we get:
$$
d_0^{-1} \left( \overline{A}_1^{(0)} \overline{A}_1^{(1)} \ldots \overline{A}_1^{(m)} \, a_{31} \, A_1^{(m)} A_1^{(m-1)} \ldots A_1^{(0)} \right) d_0 =
$$
$$
= A_2^{(0)} \, A_2^{(1)} \ldots A_2^{(m)} \, a_{31} \, \overline{A}_2^{(m)} \overline{A}_2^{(m-1)} \ldots \overline{A}_2^{(0)},~~\mbox{for}~~m \geq 1,
$$
and
$$
d_0^{-1} \left( A_1^{(-1)} A_1^{(-2)} \ldots A_1^{(m)} \, a_{31} \, \overline{A}_1^{(m)} \overline{A}_1^{(m+1)} \ldots \overline{A}_1^{(-1)} \right) d_0 =
$$
$$
= \overline{A}_2^{(-1)} \, \overline{A}_2^{(-2)} \ldots \overline{A}_2^{(m)} \, a_{31} \, A_2^{(m)} A_2^{(m+1)} \ldots A_2^{(-1)},~~\mbox{for}~~m \leq -1.
$$

Analogously, from relation $c_{12}^{-1} b_{31} c_{12} = b_{31}$ we get relations
$$
d_0^{-1} \left( \overline{B}_2^{(1)} \overline{B}_2^{(2)} \ldots \overline{B}_2^{(m)} \, b_{31} \, B_2^{(m)} B_2^{(m-1)} \ldots B_2^{(1)} \right) d_0 =
$$
$$
= B_1^{(1)} \, B_1^{(2)} \ldots B_1^{(m)} \, b_{31} \, \overline{B}_1^{(m)} \overline{B}_1^{(m-1)} \ldots \overline{B}_1^{(1)},~~\mbox{for}~~m \geq 1,
$$
and
$$
d_0^{-1} \left( B_2^{(0)} B_2^{(-1)} \ldots B_2^{(m+1)} \, b_{31} \, \overline{B}_2^{(m+1)} \overline{B}_2^{(m+2)} \ldots \overline{B}_2^{(0)} \right) d_0 =
$$
$$
= \overline{B}_1^{(0)} \, \overline{B}_1^{(-1)} \ldots \overline{B}_1^{(m+1)} \, b_{31} \, B_1^{(m+1)} B_1^{(m+2)} \ldots B_1^{(0)},~~\mbox{for}~~m \leq -1.
$$

2) Take the relation
$$
 a_{13}^{c_{12}} = a_{13}^{c_{13} c_{31}^{-1}}.
$$
Conjugating it by $c_{21}^m$ we get
$$
d_{m}^{-1} a_{13}^{c_{22}^m} d_{m} = a_{13}^{c_{13} c_{22}^m c_{31}^{-1}}.
$$
Put instead $d_m$ its expressions we get:
$$
d_0^{-1} \left( \overline{A}_1^{(0)} \overline{A}_1^{(1)} \ldots \overline{A}_1^{(m)} \, a_{13}^{c_{22}^m} \, A_1^{(m)} A_1^{(m-1)} \ldots A_1^{(0)} \right) d_0 =
$$
$$
= A_2^{(0)} \, A_2^{(1)} \ldots A_2^{(m)} \, a_{13}^{c_{13} c_{22}^m c_{31}^{-1}} \, \overline{A}_2^{(m)} \overline{A}_2^{(m-1)} \ldots \overline{A}_2^{(0)},~~\mbox{for}~~m \geq 1,
$$
and
$$
d_0^{-1} \left( A_1^{(-1)} A_1^{(-2)} \ldots A_1^{(m)} \, a_{13}^{c_{22}^m} \, \overline{A}_1^{(m)} \overline{A}_1^{(m+1)} \ldots \overline{A}_1^{(-1)} \right) d_0 =
$$
$$
= \overline{A}_2^{(-1)} \, \overline{A}_2^{(-2)} \ldots \overline{A}_2^{(m)} \, a_{13}^{c_{13} c_{22}^m c_{31}^{-1}} \, A_2^{(m)} A_2^{(m+1)} \ldots A_2^{(-1)},~~\mbox{for}~~m \leq -1.
$$

Analogously, the relation
$$
b_{31}^{c_{12}} = b_{31}^{c_{13} c_{31}^{-1}}
$$
is equivalent to the relations
$$
d_0^{-1} \left( \overline{B}_2^{(1)} \overline{B}_2^{(2)} \ldots \overline{B}_2^{(m)} \, b_{13}^{c_{22}^m} \, B_2^{(m)} B_2^{(m-1)} \ldots B_2^{(1)} \right) d_0 =
$$
$$
= B_1^{(1)} \, B_1^{(2)} \ldots B_1^{(m)} \, b_{13}^{c_{13} c_{22}^m c_{31}^{-1}} \, \overline{B}_1^{(m)} \overline{B}_1^{(m-1)} \ldots \overline{B}_1^{(1)},~~\mbox{for}~~m \geq 1,
$$
and
$$
d_0^{-1} \left( B_2^{(0)} B_2^{(-1)} \ldots B_2^{(m+1)} \, b_{13}^{c_{22}^m} \, \overline{B}_2^{(m+1)} \overline{B}_2^{(m+2)} \ldots \overline{B}_2^{(0)} \right) d_0 =
$$
$$
= \overline{B}_1^{(0)} \, \overline{B}_1^{(-1)} \ldots \overline{B}_1^{(m+1)} \, b_{13}^{c_{13} c_{22}^m c_{31}^{-1}} \, B_1^{(m+1)} B_1^{(m+2)} \ldots B_1^{(0)},~~\mbox{for}~~m \leq -1.
$$

3) Conjugating the commutativity relations by $c_{21}^m$ we get
$$
[a_{31}, a_{22}^{c_{22}^{m} c_{31}^{-m}}] = [a_{31}, a_{13}^{c_{13}^{k} c_{22}^{m-k} c_{31}^{-m}}] = [a_{22}^{c_{22}^m c_{31}^{-m}}, a_{13}^{c_{13}^k c_{22}^{m-k} c_{31}^{-m}}] = 1,
$$
$$
[b_{31}, b_{22}^{c_{22}^{m} c_{31}^{-m}}] = [b_{31}, b_{13}^{c_{13}^{k} c_{22}^{m-k} c_{31}^{-m}}] = [b_{22}^{c_{22}^m c_{31}^{-m}}, b_{13}^{c_{13}^k c_{22}^{m-k} c_{31}^{-m}}] = 1.
$$

Now we will show that $Ker (\varphi)$ is an HNN-extension with base group $T_3$ and stable letter $d_0$. Introduce subgroups $A$ and $B$ in $Ker (\varphi)$. Subgroup $A$ is generated by elements:

for $m \geq 1$
$$
 \overline{A}_1^{(0)} \, \overline{A}_1^{(1)} \ldots \overline{A}_1^{(m)} \cdot  B_2^{(m)} \, B_2^{(m-1)} \ldots B_2^{(1)},
$$
$$
\overline{A}_1^{(0)} \overline{A}_1^{(1)} \ldots \overline{A}_1^{(m)} \, X_1 \, A_1^{(m)} A_1^{(m-1)} \ldots A_1^{(0)},~~\mbox{where}~X_1 \in \{ a_{31}, a_{31}^{c_{22}^m} \},
$$
$$
\overline{B}_2^{(1)} \overline{B}_2^{(2)} \ldots \overline{B}_2^{(m)} \, Y_1 \, B_2^{(m)} B_2^{(m-1)} \ldots B_2^{(1)},~~\mbox{where}~Y_1 \in \{ b_{31}, b_{31}^{c_{22}^m} \},
$$

for $m \leq -1$
$$
A_1^{(-1)} \, A_1^{(-2)} \ldots A_1^{(m)} \cdot \overline{ B}_2^{(m+1)} \, \overline{B}_2^{(m+2)} \ldots \overline{B}_2^{(0)},
$$
$$
 A_1^{(-1)} A_1^{(-2)} \ldots A_1^{(m)} \, X_1 \, \overline{A}_1^{(m)} \overline{A}_1^{(m+1)} \ldots \overline{A}_1^{(-1)},
$$
$$
B_2^{(0)} B_2^{(-1)} \ldots B_2^{(m+1)} \, Y_1 \, \overline{B}_2^{(m+1)} \overline{B}_2^{(m+2)} \ldots \overline{B}_2^{(0)}.
$$

Subgroup $B$ is generated by elements:

for $m \geq 1$
$$
\overline{A}_2^{(-1)} \, \overline{A}_2^{(-2)} \ldots \overline{A}_2^{(m)} \cdot \overline{B}_1^{(m+1)} \overline{B}_1^{(m+2)} \ldots \overline{B}_1^{(0)},
$$
$$
A_2^{(0)} \, A_2^{(1)} \ldots A_2^{(m)} \, X_2 \, \overline{A}_2^{(m)} \overline{A}_2^{(m-1)} \ldots \overline{A}_2^{(0)},~~\mbox{where}~X_2 \in \{ a_{31}, a_{13}^{c_{13} c_{22}^m c_{31}^{-1}} \},
$$
$$
B_1^{(1)} \, B_1^{(2)} \ldots B_1^{(m)} \, b_{13}^{c_{13} c_{22}^m c_{31}^{-1}} \, \overline{B}_1^{(m)} \overline{B}_1^{(m-1)} \ldots \overline{B}_1^{(1)},~~\mbox{where}~Y_2 \in \{ b_{31}, b_{13}^{c_{13} c_{22}^m c_{31}^{-1}} \},
$$

for $m \leq -1$
$$
\overline{A}_2^{(-1)} \, \overline{A}_2^{(-2)} \ldots \overline{A}_2^{(m)} \cdot \overline{B}_1^{(m+1)} \overline{B}_1^{(m+2)} \ldots \overline{B}_1^{(0)},
$$
$$
\overline{A}_2^{(-1)} \, \overline{A}_2^{(-2)} \ldots \overline{A}_2^{(m)} \, X_2 \, A_2^{(m)} A_2^{(m+1)} \ldots A_2^{(-1)},
$$
$$
\overline{B}_1^{(0)} \, \overline{B}_1^{(-1)} \ldots \overline{B}_1^{(m+1)} \, Y_2 \, B_1^{(m+1)} B_1^{(m+2)} \ldots B_1^{(0)}.
$$

The isomorphism $\psi : A \to B$ is defined  conjugation by $d_0$ and we see that all relations of $Ker(\varphi)$, exclude the commutativity relations from 3), define this conjugation. Hence,
$Ker(\varphi)$ is an HNN-extension with base group $T_3$, stable letter $d_0$ and assotiated subgroups $A$ and $B$:
$$
Ker(\varphi) = \langle T_3, d_0 ~|~rel(T_3),~~d_0^{-1} A d_0 = B, \psi \rangle.
$$
From the properties of HNN-extension follows that the set of defining relations $rel(T_3)$ is the set of commutativity relations from 3).

\end{proof}

\end{document}